\documentclass{amsart}
\usepackage{latexsym,amssymb,amsmath}
\usepackage{amsmath,amsfonts}
\usepackage{verbatim}
\usepackage{enumitem}
\usepackage{comment}
\usepackage{esint}

\newtheorem{theorem}{Theorem}[section]
\newtheorem{lemma}[theorem]{Lemma}
\newtheorem{proposition}[theorem]{Proposition}
\newtheorem{cor}[theorem]{Corollary}

\newtheorem{remark}[theorem]{Remark}

\allowdisplaybreaks

\usepackage{pgfplots}

\usepackage{hyperref}
\hypersetup{
  colorlinks   = true,    
  linkcolor    = blue,    
  urlcolor     = blue,
  citecolor    = red     }

\newcommand{\supp}{\mathrm{supp\,}}
\newcommand{\N}{\mathbb{N}}
\newcommand{\R}{\mathbb{R}}
\newcommand{\Z}{\mathbb{Z}}
\newcommand{\T}{\mathbb{T}}

\newcommand{\lb}{\langle}
\newcommand{\rb}{\rangle}

\renewcommand{\hat}{\,\widehat}

\begin{document}

\title[Pointwise Convergence of the Schr\"odinger Flow]{Pointwise Convergence of the Schr\"odinger Flow}

\author{E. Compaan}
\address{Department of Mathematics, Massachusetts Institute of Technology} 
\email{compaan@mit.edu}

\author{R. Luc\`a}
\address{BCAM - Basque Center for Applied Mathematics, 48009 Bilbao, Spain and Ikerbasque, Basque Foundation
for Science, 48011 Bilbao, Spain.}
\email{rluca@bcamath.org} 

\author{G. Staffilani}
\address{Department of Mathematics, Massachusetts Institute of Technology}
\email{gigliola@math.mit.edu}



\date{\today}

\begin{abstract} 
In this paper we address the question of the pointwise almost everywhere limit of  nonlinear Schr\"odinger flows to the initial data, in both the continuous and the periodic settings. Then we show how, in some cases, certain smoothing effects for the non-homogeneous part of the solution can be used to upgrade to a 
uniform convergence to zero of this part, and we discuss the sharpness of the results obtained. We also use randomization techniques to prove that with much less regularity of the initial data, both in continuous and the periodic settings, almost surely one obtains uniform  convergence of the nonlinear solution  to the initial data, hence showing how more {\it generic} results can be obtained.
\end{abstract}

\maketitle
\tableofcontents

\section{Introduction}

In this work, we are concerned with the question of almost everywhere convergence of solutions to certain  nonlinear Schr\"odinger equations (NLS) to initial data.
More precisely, let $u(x, t)$ be a solution to  
\begin{equation}\label{NLSeq} 
\begin{cases} i \partial_t u + \Delta u =  \mathcal{N} (u) ,
\\ u(x,0) = f(x), 
\end{cases} 
\quad x \in \T^d \text{ or } \R^d \, ,
\end{equation} 
where $\mathbb{T} := \R / 2\pi \mathbb{Z}$ and $\mathcal{N}$ is a power type nonlinearity.
 If $f \in H^s$, for what $s$ do we have that  $u(x,t) \to f(x)$ as $t \to 0$ for (Lebesgue) almost every $x$\,?

In the linear Euclidean setting, namely $\mathcal N = 0$ and $x \in \R^{d}$, this question was first 
posed by Carleson \cite{Carleson}, who showed that almost everywhere (a.e.) convergence holds for $f \in H^\frac14(\R)$. 
Dahlberg--Kenig \cite{DahlbergKenig} showed that this one dimensional result is sharp; in fact they proved that 
$s\geq \frac14$ is a 
necessary condition for a.e. convergence on~$\R^d, d\geq 1$. 
Since then, the higher dimensional problem has been studied by many authors 
\cite{MR729347, MR848141, MR904948, Vega, MR1194782, MR1671214, MR1748920, MR1748921, MR2033842, MR2264734, Bourgain2013, Luc2018, 1608.07640, MR3613507, MR3842310}. 
Recently,
Bourgain \cite{Bourgain2016} proved that
$s \geq \frac{d}{2(d+1)}$ is a necessary condition for a.e. pointwise convergence to the data (see also \cite{MR3903115} for an alternative counterexample). 
This has been proved to be sharp, up to the endpoint, by Du--Guth--Li \cite{DGL} in the $\R^2$ case, and by Du-- Zhang \cite{Du2019} in 
higher dimensions. 

In the linear periodic setting, namely $\mathcal N = 0$ and $x \in \T^{d}$, much less is known. The only result appears to be that of Mouya--Vega \cite{MR2409184}  when $d=1$,
(sufficiency of $s > \frac13$ and necessity of $s \geq \frac14$), 
which method of proof, based on Strichartz estimates, has been extended to higher dimensions by Wang--C. Zhang \cite{WangZhang}. Together with recent 
improvements in periodic Strichartz estimates \cite{BourgainDemeter}, one can show that $s > \frac{d}{d+2}$ is a sufficient condition for almost everywhere 
convergence to initial data\footnote{Although in this paper  we only consider rational tori, this particular result 
holds for any torus since it is based on Strichartz estimates, now available for any torus thanks to \cite{BourgainDemeter}.}. 
We refer to Section \ref{Deterministic ResultsLinear} for more details. In Section  \ref{Deterministic ResultsLinear} we 
also show that almost everywhere convergence fails when  
$s<\frac{d}{2(d+1)}$,\footnote{This follows adapting the non periodic counterexamples to the periodic setting and still works if we consider irrational tori.} 
see Proposition \ref{couter}.   
At the  moment, in the periodic  case almost sure convergence when $s\in \left[ \frac{d}{2(d+1)},  \frac{d}{d+2} \right]$ remains  an open question.

In the first part of this paper we extend these results to the nonlinear setting. Hereafter 
$\Omega$ denotes either $\T$ or $\R$.
We define
\begin{equation}\label{Def:SOmegaD}
s_{\Omega^d} :=   \begin{cases}
 \frac{d}{d+2} & \text{ if }  \Omega = \mathbb{T}, \vspace{6pt}
 \\  
  \frac{d}{2(d+1)} & \text{ if }  \Omega = \mathbb{R} .
  \end{cases}
\end{equation}  
Summarizing  the results mentioned above, one has
$$
\lim_{t \to 0} e^{it\Delta}f(x) = f(x) \quad \mbox{for a.e. $x \in \Omega^d$}
$$
for all $f \in H^{s}(\Omega^d)$ with $s > s_{\Omega^d}$. If $\Omega^{d} = \R$ we only need $s \geq s_{\R} = \frac14$. 

In the following theorem we prove that a similar result is true for solutions to NLS with power nonlinearities.

\begin{theorem}\label{MainTHM1}
Let $\mathcal{N}(z) = \pm |z|^{p-1}z$ with $p \geq 3$.
If $f \in H^{s}(\Omega^d)$ with 
\begin{equation}\label{RegAssTheorem1}
s > \max \left( s_{\Omega^d}, \ \frac{d}{2} - \frac{2}{p-1} \right) \, ,
\end{equation}
and $u$ is the corresponding solution to \eqref{NLSeq}, then
\begin{equation}\label{AlmostEverywhereConv}
\lim_{t \to 0} u(x, t) = f(x) \quad \mbox{for a.e. $x \in \Omega^d$} \, .
\end{equation}
If $\Omega^d = \R$ and $p < 9$ we can relax the condition $s > s_{\R} = \frac14$ to $s \geq \frac14$.
Moreover if we consider the cubic equation ($p=3$) and $d=1$ or $\Omega^d = \R^2$ we have for~$s > \frac d6$
\begin{equation}\label{EverywhereConv}
\lim_{t \to 0} u(x, t) - e^{it \Delta} f(x) = 0 \quad \text{for every $x \in \Omega^d$} \, 
\end{equation} 
and the convergence is uniform with respect to the $x$ variable\footnote{A proof of \eqref{EverywhereConv} in the case $d=1$ is in \cite{ErdoganTzirakis, Comp2013}.
Here we extend 
the result to $\Omega^d = \R^2$.}, namely
$$
\lim_{t \to 0} \sup_{x \in \Omega^d} |u(x, t) - e^{it \Delta} f(x)| = 0 \, .
$$
\end{theorem}

\begin{remark}\label{Remark1}
The result is sharp in the following sense. The conditions $p \geq 3$ and
$$
s > \max \left( 0, \ \frac{d}{2} - \frac{2}{p-1} \right) \, , 
$$
ensure that the flow is locally well defined, in fact $s_c:=\frac{d}{2} - \frac{2}{p-1}$ is the critical exponent. The extra assumption~$s > s_{\Omega^d}$ 
ensures that the linear part $e^{it\Delta} f$ of the flow converges pointwise a.e. to the initial datum $f$. This 
condition is sharp if $\Omega = \mathbb{R}$ (modulo endpoints when $d \geq 2$) and we do not expect improved convergence 
to the data when we introduce a nonlinearity. Moreover by the proof of Theorem \ref{MainTHM1} it 
will be clear that any improvement of the exponent $s_{\Omega^d}$ into the linear setting would provide 
an analogous improvement of Theorem \ref{MainTHM1} as well. More precisely, if we define 
\begin{equation}\label{Inf}
s^*_{\Omega^{d}} := \inf \left\{ s \, : \,  \lim_{t \to 0} e^{it\Delta}f(x) = f(x) \quad \mbox{for a.e. $x \in \Omega^d$}, \quad \mbox{$\forall f \in H^{s}(\Omega^d)$} \right\} \, , 
\end{equation}
we can replace the assumption \eqref{RegAssTheorem1} by 
\begin{equation}
s > s^*_{\Omega^d} \quad \mbox{and} \quad s >  \frac{d}{2} - \frac{2}{p-1}  \, ,
\end{equation}
and we can relax $s > s^*_{\Omega^d}$ to $s \geq s^*_{\Omega^d}$ if the $\inf$ in \eqref{Inf} is a $\min$.
\end{remark}

In the  second part of this paper we adopt a different approach than the purely deterministic one we presented in the first part  that culminated  in Theorem \ref{MainTHM1}.  More precisely we consider the linear and  the cubic NLS and we show that actually uniform convergence to the  
data is {\it generically} true for 
initial data which are less smooth than the data postulated in Theorem \ref{MainTHM1}. In the periodic setting, we consider  
\begin{equation}\label{ReprInitialData}
 f^\omega(x) = \sum_{n\in \Z^d} \frac{g_n^\omega}{\lb n \rb^{\frac{d}2+\alpha}} e^{in \cdot x} \, ,
 \qquad \alpha >0 \, , 
 \end{equation}
where $g_n^\omega$ are independent (complex) standard Gaussian variables and we define $\lb \cdot \rb = (1 + |\cdot|^2)^\frac12$.
We will need the following facts, proved in Section \ref{linSchrTd}. Fix $t \in \R$. Then $ e^{it\Delta}f^\omega(x)$ belongs to $\bigcap_{s < \alpha}H^{s}(\T^d)$ 
$\mathbb P$-almost surely.  
Thus we are working at the $H^{\alpha-}$ level.  
Moreover, $e^{it \Delta}f^{\omega}$ is $\mathbb{P}$-almost surely a continuous function of the $x$ variable, where $\mathbb{P}$ is the law of the sequence 
$\{ g_n^\omega \}_{n \in \mathbb{Z}}$.

In the following statement we consider the Wick ordered cubic NLS in $\T^d, \, d=1, 2$. Namely equation \eqref{NLSeq} with nonlinearity
\begin{equation}\label{eq:wickNLS}
\mathcal{N}(u) := \pm u \left( |u|^2 - 2 \mu \right), 
\qquad 
\mu := \fint_{\T^d} |u(x, t)|^2 dx \, ,
\end{equation}
Since once we fix the initial datum 
$f\in L^2$, solutions to this equation are related to that of the cubic NLS by multiplication with a factor $e^{i 2 \mu t}$, 
the study of pointwise  
convergence of \eqref{eq:wickNLS} turns out to be equivalent to that of the cubic NLS.

\begin{theorem}\label{MainTHM2}
Let $f^{\omega}$ be defined in \eqref{ReprInitialData} for $\alpha>0$.
Then one has $\mathbb{P}$-almost surely the following. 
For all $t \in \R$ the free solution 
$e^{it \Delta} f^\omega$ belongs to $\bigcap_{s < \alpha}H^{s}(\T^d)$ 
and is continuous in the $x$ variable. 
Moreover 
\begin{equation}\label{r2prob-lin}
\lim_{t \to 0} e^{it\Delta}f^{\omega}(x) = f^{\omega}(x) \quad \text{ for every $x \in \T^d$} \, .
\end{equation}
and the convergence is uniform in the $x$ variable.
Let $d=1, 2$ and let $u$ be the solution to the Wick ordered cubic NLS \eqref{eq:wickNLS} with  random initial data 
$f^\omega$ as above.  Again, $\mathbb{P}$-almost surely one has
\begin{equation}\label{AlmostEverywhereConvergenceRandom}
\lim_{t \to 0} u(x, t) = f^{\omega}(x) \quad \mbox{for a.e. $x \in \T^d$} \, .
\end{equation}
Furthermore, if $\alpha > \frac{d-1}2$, then 
\begin{equation}\label{EverywhereConvergenceRandom}
\lim_{t \to 0} u(x, t) - e^{it\Delta}f^{\omega}(x) = 0 \quad \mbox{for every $x \in \T^d$}.  
\end{equation}
and the convergence is uniform in the $x$ variable.
\end{theorem}

\begin{remark}
Notice that if $d=1$ combining \eqref{r2prob-lin} and \eqref{EverywhereConvergenceRandom} we get in fact a stronger convergence statement than \eqref{AlmostEverywhereConvergenceRandom}, namely the convergence occurs at any $x$ and uniformly ($\mathbb{P}$-almost surely). 
If $d=2$ the combination of \eqref{r2prob-lin} and \eqref{EverywhereConvergenceRandom} gives 
this stronger convergence result only for data that are in $H^{\frac12 +}(\T^2)$, while by \eqref{AlmostEverywhereConvergenceRandom}
we see that 
a.e. convergence occurs
for initial data that are merely in $H^{0 +}(\T^2)$ ($\mathbb{P}$-almost surely). 
\end{remark}

We also obtain results for randomized initial data on Euclidean spaces. 
We use an integer tiling--type randomization, of the type introduced in \cite{ZhangFang, LuhrMend, BOP15}. To begin, we construct a partition of unity on $\R^d$. 
Let $\eta$ be a smooth cut-off of the unit interval. Specifically, let $\eta: \R^d \to [0,1]$ be a smooth function such that $\supp \eta \subset \{ \xi : |\xi| \leq 2\}$ and $\eta(\xi) = 1$ for all $|\xi| \leq 1$. Then for $n \in \Z^d$, define 
\begin{equation}\label{eq:psiDef}
\psi_n(\xi) = \frac{\eta(\xi - n)}{\sum_{\ell \in \Z^d} \eta(\xi-\ell)}. 
\end{equation}
Observe that $\psi_n$ is  smooth function supported on $\{ \xi: |\xi - n| \leq 2\}$ and we have $\sum_n \psi_n(\xi) = 1$ for all $\xi$. 

Fix $f \in H^s(\R^d)$ with $s > 0$. We construct a randomization $f^\omega$ of $f$ as follows. Let $g_n^\omega$ be a collection of 
independent (complex) standard Gaussian variables and define $f^\omega$ by
\begin{equation} \label{eq:RdRand}  
\hat{f^\omega}(\xi) = \sum_{n \in \Z^d} g_n^\omega\, \psi_n(\xi) \hat{f}(\xi) \, .
\end{equation}
This randomization satisfies analogous properties to the one described above in the periodic setting. If 
$f^{\omega} \in H^{s}(\R^{d})$ then $e^{it \Delta} f^{\omega}$ is in $H^{s}$ and is a continuous function of the $x$ variable 
$\mathbb{P}$-almost surely,
for all $t \in \R$; we refer to Section \ref{SectionlinRandRd} for more details.  
In order to compare \eqref{eq:RdRand} with \eqref{ReprInitialData} it is convenient to look at the Fourier transform of \eqref{ReprInitialData}, that is 
\begin{equation}\label{ReprInitialDataFourier}
\hat{ f^\omega}(n) =  \frac{g_n^\omega}{\lb n \rb^{\frac{d}2+\alpha}} 
\end{equation}
and remember that $\psi_n(\xi)$ is a bump function of a neighborhood of $\xi = n$ and that a function with 
Fourier coefficients $\lesssim \lb n \rb^{-\frac{d}2-\alpha}$ Belongs to $H^{\alpha-}$. 
We then have the following 

\begin{theorem}\label{MainTHM3}
Fix $f \in H^s(\R^d)$ with $s > 0$. Let $f^{\omega}$ be a randomization of $f$ as defined in \eqref{eq:RdRand}.
Then one has $\mathbb{P}$-almost surely the following.
For all $t \in \R$ 
the free solution $e^{it \Delta} f^\omega$ belongs to $H^{s}(\R^d)$ 
and that is continuous in the $x$ variable. 
Moreover
\begin{equation}\label{r2prob-lin2}
\lim_{t \to 0} e^{it\Delta}f^{\omega}(x) = f^{\omega}(x) \quad \text{for every} \, \, x \in \R^d
\end{equation}
and the convergence is uniform in the $x$ variable.
Let then $u$ be the solution to the cubic NLS with initial data $f^{\omega}$.
If $d=1,2$ and~$s > \frac{d}{6(d+1)}$, again $\mathbb{P}$-almost surely
\begin{equation}\label{r2prob-nonlin} 
\lim_{t \to 0} u(x,t) = f^{\omega}(x) \quad \mbox{ for almost every $x \in \R^d. $}
\end{equation} 
\end{theorem}

\begin{remark}
Both the randomization procedures described do not improve smoothness; see for example Remark 1.2 in \cite{BurqTzvet} and the introduction of \cite{LuhrMend}.  
\end{remark}

\begin{remark}
The statements \eqref{r2prob-lin}, \eqref{r2prob-lin2} are still true, with same proof, as explained in Sections \ref{linSchrTd}, \ref{SectionlinRandRd}, if we consider random Fourier series drawn from distribution with sufficiently strong decay properties.
In fact, the argument we present works 
for independent $g_n^\omega$ drawn 
from a sequence of centered sub Gaussian random variables with unitary variance. 
\end{remark} 

\begin{remark}
The randomization precedures give convergence $e^{it\Delta} f^{\omega} \to f^{\omega}$   
for {\it any} $x \in \Omega^d$ (and uniformly) $\mathbb{P}$-almost surely; see \eqref{r2prob-lin}, \eqref{r2prob-lin2}. In the deterministic case one has convergence at {\it any} $x \in \R^d$ (and uniformly) only for $s > d/2$. 
In fact when $s \leq d/2$ it make sense to consider a refined version of the a.e. convergence problem, which consists into determine 
the (worst possible) Hausdorff dimension of the set 
where the convergence to $H^{s}(\R^d)$ initial data fails. This problem, introduced in \cite{MR997967}, has been solved
in \cite{MR2754999} when $s \in \left[ \frac d4, \frac d2 \right]$. In the range $s \in \left(\frac{d}{2(d+1)}, \frac d4\right)$, where the problem is still open, 
the best positive and negative results to date are 
in \cite{Du2019} and \cite{MR3613507, MR3903115}, respectively. 
\end{remark}

We now give  a brief description of our methods of proof for the three main theorems listed above. 

To prove \eqref{AlmostEverywhereConv}  in Theorem \ref{MainTHM1} we consider (smooth) approximations of the solutions of NLS obtained by truncating \eqref{NLSeq} on 
the first $N$ Fourier modes. Since for this class of solutions we have pointwise convergence to the initial data, we are able to rewrite the convergence 
problem as an $L^{2}_{x, loc}$ bound for a suitable maximal function, adapted to the nonlinear setting; see Proposition \ref{MaxEstObv}. It is worth  mentioning that 
(already in the linear setting) 
the maximal function approach is 
the most powerful tool to study a.e. pointwise convergence to the initial data.
In order to obtain a good enough bound, in Proposition \ref{MaxEstObv} we embed the restriction space $X^{s, \frac12+}_{\delta}$ into the space
$$\left\{ F(x,t) : \|\sup_{0 \leq t \leq \delta} |F(x,t)| \|_{L^{2}_{x, loc}} < \infty \right\} \, , $$ 
for $s > s_{\Omega^d}$; see Proposition \ref{MainLemma} ($[0, \delta]$, $\delta >0$ will be the local existence time). 
This embedding allows us to use Strichartz estimates to 
conclude the proof; see Section~\ref{Sec:ProofOfTheorem1}.

For the cubic NLS we can prove stronger results, taking advantage of the algebraic structure of the nonlinearity
$\mathcal{N}(z) = \pm |z|^{2}z$. 
A first example of this phenomenon is already in the statement \eqref{EverywhereConv}. Let us consider $x \in \R^2$ to fix the notations 
(similar observations can be made if $x \in \Omega$). 
Since for $s > s_{\R^2} = \frac13$ one 
has $e^{it\Delta} f(x) \to f(x)$ as $t \to 0$ for a.e. $x \in \R^2$, we see that \eqref{EverywhereConv} is 
clearly stronger than~\eqref{AlmostEverywhereConv}. In fact, we can show that for any $t \in \R$, the function
$$
x \in \R^2 \to u(x, t) - e^{it\Delta}f(x) \in \mathbb{C}
$$ 
is continuous. Moreover the map
\begin{equation}\label{2dSmoothing}
t \in \R \to u(x, t) - e^{it\Delta}f \in C_x(\R^2) \quad  \mbox{(with the $\sup_{x \in \R^2}$ norm)} 
\end{equation}
is also continuous.
This stronger convergence result is a consequence of a smoothing effect associated to the cubic nonlinearity on $\R^2$, that we prove 
in Corollary \ref{uuu}. A similar smoothing effect has been noted in $1d$ in both the periodic\footnote{In fact in the periodic setting one has 
to consider the Wick ordered equation; 
see the paragraph before Theorem \ref{MainTHM2} and Section \ref{WickOrderedNLS} for details.}
 \cite{ErdoganTzirakis} and non-periodic \cite{Comp2013} settings:
the nonlinearity turns out to be $\sigma$--smoother than the initial datum in $H^s$, with $\sigma < \max(2s, \frac12)$. 
We strengthen and extend this smoothing effect on $\R^d$, $d=1,2$ to $\sigma < \max(2s, 1)$; see Corollary \ref{uuu}. 
Using these facts and the 
Sobolev embedding $H^{\frac12+}(\Omega) \hookrightarrow L^{\infty}(\Omega)$, we see that the ($1d$ analog of) property~\eqref{2dSmoothing} 
is satisfied by initial data in $H^{s}(\Omega)$ with $s > \frac16$. On the other hand, in \cite{MR2337486} (see also \cite{MR3903115}) it has been observed that for any $s < \frac14$ 
there are initial data such that $\limsup_{t \to 0} |e^{it\Delta} f(x)| = \infty$ for~$x$ in a set of strictly positive measure. This construction, done
for $\Omega = \R$, is based on the Dahlberg--Kenig
counterexample and can be repeated also in the periodic setting. Combining this with the above mentioned 
smoothing effect of the $1d$ cubic NLS, it is immediate to see that the convergence statement \eqref{AlmostEverywhereConv} 
fails for $s \in (\frac16, \frac14)$ for the one-dimensional cubic NLS ($p=3$). However, as observed in 
Remark \ref{Remark1}, we expect \eqref{AlmostEverywhereConv} to fail for any~$s < s^*_{\Omega^d}$ and for any $p \geq 3$.

The proofs of Theorem \ref{MainTHM2} and  Theorem \ref{MainTHM3} rely upon a combination of the following two facts. First, the randomization improves 
the integrability of the 
randomized function. This allows us to deduce uniform  convergence of the linear propagator $e^{it\Delta}f^{\omega}$ to the 
initial data $f^{\omega}$, $\mathbb{P}$-almost surely; see Propositions \ref{RandomLinCOnv1} and \ref{RandomLinCOnv2}. 
Second, we deduce a smoothing effect associated to the cubic nonlinearity. This allows us to 
control the nonlinear (Duhamel) contribution $u(x,t) - e^{it \Delta}f^{\omega}$. While in the Euclidean case (Theorem \ref{MainTHM3}), we use the  deterministic smoothing effect given in Corollary \ref{uuu}, in the periodic case (Theorem \ref{MainTHM2}) the proof is much more involved. In fact, we follow the argument of Bourgain  in \cite{bourgain1996invariant} and we start with the  Wick--reordering of  the nonlinearity. Then using more probabilistic arguments and Jarnick's theorem (namely counting lattice points on convex archs),  
as in \cite{bourgain1996invariant}, we obtain in Proposition \ref{Bourgain1/2} a precise quantification of  the amount of smoothing
(that in this case happens $\mathbb{P}$-almost surely). In our argument a quantification of the smoothing is necessary because 
we need to be sure that  the nonlinear (Duhamel) contribution sits    in $X^{s, \frac12 +}_{\delta}$ with $s > s_{\T^d}$ (we consider $d=1,2$), so that we can conclude the proof by
implementing techniques from Theorem \ref{MainTHM1}.

\subsection{Acknowledgements}

E. Compaan is supported by NSF MSPRF 1704865. 
R. Luc\`a is supported by the ERC grant 676675 FLIRT, by BERC 2018-2021, by 
BCAM Severo Ochoa SEV-2017-0718 and IHAIP project PGC2018-094528-B-I00 (AEI/FEDER, UE).
G. Staffilani is supported by NSF grants DMS 1462401 and DMS 1764403.
The authors thank Chenjie Fan for useful discussions on Propositions \ref{RandomLinCOnv1} and \ref{RandomLinCOnv2} and the referees for their 
useful comments.

\subsection{Notations and terminology} 
For a fixed $p\in \R$ we often use the notation~$p+:=p+\varepsilon$, $p-:=p-\varepsilon$, where~$\varepsilon$ is any 
sufficiently small strictly positive real number. 
When in the same inequality we have two such quantities we use the following notation to compare them. 
We write $p+ \dots + := p + \varepsilon \cdot (\mbox{number of $+$})$,
$p- \dots - := p - \varepsilon \cdot (\mbox{number of $-$})$.
We will use $C >0$ to denote several constants depending only on fixed parameters, like 
for instance the dimension $d$. The value of $C$ may clearly differ from line to line.
Let~$A, B >0$. We may write~$A \lesssim B$ if~$A \leq C B$ when~$C > 0$ is such a constant. 
We write~$A \gtrsim B$ if~$B \lesssim A$ and~$A \sim B$ when~$A \lesssim B$ and~$A \gtrsim B$. 
We write~$A \ll B$ if~$A \leq c B$ for~$c >0$ sufficiently small (and depending only on fixed parameters) 
and~$A \gg B$ if~$B \ll A$. We denote~$A \wedge B := \min (A, B)$ 
and~$A \vee B := \max (A, B)$.  We refer to the following inequality
$$
\| D^s P_N f \|_{L^{q}} \lesssim N^{s + \frac{d}{p} - \frac{d}{q} } \|  P_N f \|_{L^p}, 
\quad 
1 \leq p \leq q \leq \infty \, , 
$$ 
simply as Bernstein inequality. Here $P_{N}$ is the frequency projection on the annulus~$\xi \sim N$.

\section{Preliminaries}\label{prelim}

Let us recall that $\Omega$ denotes either $\T$ or $\R$. We denote by $B_{\rho}$ a ball of radius $\rho >0$ centered at a generic point of $\Omega^d$ or $\Z^d$.
The following Strichartz estimates are the main 
tool to study 
the nonlinear Schr\"odinger flow:
\begin{equation}\label{FullStrich}
\|e^{it\Delta} f(x)\|_{L^{p}_{x,t}(\Omega^{d+1})} \lesssim N^{\frac{d}{2} - \frac{d+2}{p} + }  \|  f \|_{L^2_{x}(\Omega^d)}, 
\quad 
p \geq 2 \left(\frac{d+2}{d}\right),
\quad 
\supp \hat{f} \subseteq B_N \, .
\end{equation}
These estimates were proved in \cite{MR512086} for $\Omega = \R$ and in \cite{BourgainDemeter} for $\Omega = \T$.  
The additional factor $N^{0+}$ is removable except when $\Omega = \T$ and $p = 2 \left(\frac{d+2}{d}\right)$; see \cite{MR1209299, MR3512894} and the references 
therein. However, we never use this finer information.  

Hereafter $\delta \in (0,1]$, sometimes we will restrict to sufficiently small values of $\delta$. The main tool used in the study of the a.e. pointwise convergence of solutions to linear Scr\"odinger equation 
to the initial data is the following maximal estimate
\begin{equation}\label{OmegaMaxEst}
\left\| \sup_{0 \leq t \leq \delta} |e^{it\Delta } f(x) | \right\|_{L^2_x(B_1)} \lesssim \| f \|_{H^{s}_x(\Omega^d)} \, .
\end{equation}
The validity of this estimate is equivalent to the fact that $e^{it\Delta } f(x) \to f(x)$ as $t \to 0$ for almost every (with respect to the Lebesgue measure) $x \in B_1$. One implication of this statement 
is
elementary; the other is a consequence of the Stein--Niki\v{s}hin maximal principle \cite{10.2307/1970308, MR0343091}.
Inequality \eqref{OmegaMaxEst}
holds for all $s > s_{\Omega^d}$ where $s_{\Omega^d}$ is defined in \eqref{Def:SOmegaD}; see the introduction and
the forthcoming Proposition \ref{MaxInTorusPartial}.

The main result of this section (Lemma \ref{MainLemma}) is that given a function 
$$F: (x, t) \in \Omega^{d} \times \R \to F(x, t) \in \mathbb{C}$$
we can bound the 
$L^{2}_x(B_1)$ norm of the associated maximal function $\sup_{0 \leq t \leq \delta} |F(x,t)|$ with an appropriate~$X^{s,b}_{\delta}$ norm of $F$; see also \cite{MR1209299}. This embedding is used to obtain pointwise convergence results for solutions to the nonlinear Schr\"odinger equation.

We recall that 
\begin{equation}\nonumber 
\| F \|_{X^{s, b}_{\delta}} := \inf_{ G = F \ \mbox{on} \ t \in [0, \delta]} \| G \|_{X^{s, b}},
\end{equation}
where 
$$
 \| F \|^2_{X^{s, b}} 
 :=  \int_{\R} \sum_{n \in \mathbb{Z}^{d}}  \langle \tau + |n|^2 \rangle^{2b} \langle n \rangle^{2s} | \widehat{F}(n, \tau)|^2 d \tau  
\quad \mbox{if} \quad  \Omega = \T \, ,
$$
 $$
  \| F \|^2_{X^{s, b}} :=  \int_{\R} \int_{\R^d}  \langle \tau + |\xi|^2 \rangle^{2b} \langle \xi \rangle^{2s} | \widehat{F}(\xi, \tau)|^2  d \xi d \tau
  \quad \mbox{if} \quad   \Omega = \R \, , 
$$
$\langle \cdot \rangle := (1 + |\cdot|^2)^{\frac12}$, and $\widehat{F}$ is the space-time Fourier transform of $F$.

The next lemma shows how to embed $X^{s,b}_{\delta}$, $b > \frac12$, into several functional spaces. The proof can be found in \cite[Lemma 2.9]{tao2006nonlinear}, 
in the case $\Omega = \R$. 
The argument adapts to $\Omega = \T$.

\begin{lemma}\label{MainLemmaGeneralized}
Let $b > \frac12$ and let $Y$ be a Banach space of functions 
$$F: (x, t) \in \Omega^{d} \times \R \to F(x, t) \in \mathbb{C} \, .$$
Let $\alpha \in \R$. Assume
\begin{equation}\label{LMERN}
\|  e^{i \alpha t} e^{it\Delta} f(x)    \|_{Y} \leq  C \| f  \|_{H^s(\Omega^d)} \, ,
\end{equation}
with a constant $C >0$ uniform over $\alpha \in \R$. Then
$$
\| F  \|_{Y} \leq C \| F \|_{X^{s,b}} \, .
$$
\end{lemma}
Using Lemma \ref{MainLemmaGeneralized} with
$$
\| F \|_{Y} = \left\| \sup_{0 \leq t \leq \delta} | F(x,t) |  \right\|_{L^{2}_x(B_1)}
$$
and the fact that the maximal estimate \eqref{OmegaMaxEst} hold for $s > s_{\Omega^d}$ ($s \geq \frac14$ if $\Omega^d = \R$), we have the following 
\begin{lemma}\label{MainLemma}
Let $b > \frac12$ and $s > s_{\Omega^d}$ defined in \eqref{Def:SOmegaD}. We have
\begin{equation}\label{BouMaxEmbedding}
\left\| \sup_{0\leq t \leq \delta} | F(x,t) |  \right\|_{L^{2}_x(\Omega^d)} \lesssim 
\| F \|_{X^{s, b}_{\delta} }  \, .
 \end{equation}
If $\Omega^d = \R$ we can relax $s > s_{\R} = \frac14$ to $s \geq \frac14$.
\end{lemma}

We also recall several useful estimates, repeatedly used in the paper, that can be obtained using the Strichartz estimates \eqref{FullStrich} and Lemma \ref{MainLemmaGeneralized}.
Let $\operatorname{P}_{N}$ be the frequency projection into the annulus of size $N$, namely $\{ N/2 < |\xi| \leq N \}$.   
Let $\operatorname{P}_{A}$ be the frequency projection into the set $A$. We denote by $Q_N$ a (frequency) cube of side length $N$, centered at any point. Then in the statement of the lemma below we use
$\Gamma_{N}$ to denote either $Q_N$ or the (frequency) annulus of size $N$. Thus $\operatorname{P}_{\Gamma_N}$ is either 
$\operatorname{P}_{Q_N}$ or $\operatorname{P}_{N}$. 
\begin{lemma}
Let 
$p \geq 2 \left(\frac{d+2}{d}\right)$. Then
\begin{equation}\label{ITFCOS2}
\| \operatorname{P}_{\Gamma_N} F \|_{L^{p}_{x, t}(\Omega^{d+1})} \lesssim  
N^{ \frac{d}{2} - \frac{d+2}{p} - s +} \| \operatorname{P}_{\Gamma_N} F \|_{X^{s, \frac12 +}} \, ,
\end{equation}
\begin{equation}\label{Dual1}
\|  \operatorname{P}_{\Gamma_N} F \|_{X^{0, -\frac12++}} 
\lesssim N^{0+}  \| \operatorname{P}_{\Gamma_N} F \|_{L^{2 \frac{d+2}{d+4}+}_{x, t}(\Omega^{d+1})}   \, .
\end{equation}
Let also $s > \frac{d}{2} - \frac{d+2}{p}$, then
\begin{equation}\label{ITFCOS2Summed}
\|  F \|_{L^{2\frac{d+2}{d-2s}-}_{x, t}(\Omega^{d+1})} \lesssim  \| F \|_{X^{s, \frac12+}} 
 \,.
\end{equation}
\end{lemma}
\begin{proof}
By dyadic decomposition in frequency, summing a geometric series, and then using Plancherel, we see that \eqref{FullStrich} implies
\begin{equation}\nonumber
 \| e^{it\Delta} f\|_{L^{p}_{x,t}(\Omega^{d+1} )} \lesssim   \| f \|_{H^{\frac{d}{2} - \frac{d+2}{p} +}(\Omega^d)} \, .
\end{equation}
This and Lemma \ref{MainLemmaGeneralized} imply 
\begin{equation}\label{ITFCOS2Preq}
\|  F \|_{L^{p}_{x, t}(\Omega^{d+1})} \lesssim   \| F \|_{X^{\frac{d}{2} - \frac{d+2}{p} + , \frac12 +}} \, .
\end{equation}
Now the estimate \eqref{ITFCOS2} is an immediate consequence of \eqref{ITFCOS2Preq}.
Letting $p=2 \left(\frac{d+2}{d}\right)$ in \eqref{ITFCOS2} and interpolating it with $\| \cdot \|_{X^{0,0}} = \| \cdot \|_{L^{2}_{x,t} (\Omega^{d+1})}$ we get
\begin{equation}\nonumber
\| \operatorname{P}_{\Gamma_N} F \|_{L^{2 \frac{d+2}{d} -}_{x, t} (\Omega^{d+1})} \lesssim    
N^{0+} \| \operatorname{P}_{\Gamma_N} F \|_{X^{0, \frac12 -- }} \, .
\end{equation}
Dualizing this yields \eqref{Dual1}. 
The estimate \eqref{ITFCOS2Summed} follows from
\eqref{ITFCOS2} by taking $p=2\left(\frac{d+2}{d-2s}\right)-$ and $\operatorname{P}_{\Gamma_N} = \operatorname{P}_N$, 
after again performing a dyadic frequency decomposition, summing a geometric series, and using Plancherel.

\end{proof}
We also recall some well known properties of the $X^{s, b}$ spaces that are repeatedly used in the paper; see for example \cite{MR1357398}.
Hereafter $\eta$ is a smooth cut-off of the unit interval. 
\begin{lemma}
Let $s \in \R$. 
Then
\begin{equation}\label{Basic1}
\| \eta (t) e^{it \Delta} f(x)   \|_{X^{s,  \frac12+}} \lesssim \| f \|_{H^{s}(\Omega^d)} \, ,
\end{equation}
\begin{equation}\label{Basic3} 
\left\| \eta(t)  \int_0^t e^{i(t-t') \Delta} F(\cdot,t')   dt'   \right\|_{X^{s,  \frac12 + }}
\lesssim \|   F     \|_{X^{s,  - \frac12 + }} \, ,
\end{equation}
\begin{equation}\label{Basic2} 
\|  F  \|_{X^{s, - \frac12 + }_{\delta}}
\lesssim \delta^{0+} \|  F  \|_{X^{s, - \frac12 ++}_{\delta}} \, .
\end{equation}
\end{lemma}
We end this section with a bilinear estimate in the space $\R^2$ and in $\R$. We first prove the result in $\R^2$. This is the harder case, and a proof has already appeared  in  \cite{MR1828607}; we report it below for completeness.  The analogous result in $\R$ is easier and can be proved with similar techniques. These blinear estimates are used to obtain smoothing results for the nonlinear (Duhamel) part of the solution.

 We start with  some notation. For dyadic numbers 
$M_0, M_1, M_2$ we set $M_*=\min (M_0, M_1, M_2)$ and $M^*=\max (M_0, M_1, M_2)$. We use the notation
$f\chi_{\{|\mu|\sim M\}}=:f_M$ for the restriction to a dyadic annulus. We then define
$$\int_{\tau_0+\tau_1+\tau_2=0, \\ \mu_0+\mu_1+\mu_2=0}=:\int_{*}$$
and 
\begin{equation}\label{C+-}
C_{\mp}(f;g,h)=\int_{*}f(\mu_0,\tau_0)\frac{g(\mu_1,\tau_1)}{(1+|\tau_1-|\mu_1|^2)^b}\frac{h(\mu_2,\tau_2)}{(1+|\tau_2\pm|\mu_2|^2)^b} d \tau_1 d \tau_2 d \mu_1 d \mu_2,
\end{equation}
with $b > \frac12$.

\begin{lemma}[{\cite[Lemma 1]{MR1828607}}]\label{bilinear}  Assume we are in $\R^2$.
The following estimates hold
$$|C_{+}(f_{M_0};g_{M_1},h_{M_2})|\lesssim \left(\frac{M_*}{M^*}\right)^{\frac12}\|f_{M_0}\|_{L^2}\|g_{M_1}\|_{L^2}\|h_{M_2}\|_{L^2},$$
$$|C_{-}(f_{M_0};g_{M_1},h_{M_2})|\lesssim \left(\frac{M_1\wedge M_2}{M_1\vee M_2}\right)^{\frac12}\|f_{M_0}\|_{L^2}\|g_{M_1}\|_{L^2}\|h_{M_2}\|_{L^2}.$$
\end{lemma}
We then have the following corollary.
\begin{cor}\label{uuu} Assume we are in $\R^d$, $d=1,2$,  $b>b'>\frac12$ and $s\geq 0$. Then if $\sigma < \min (2s, 1)$ we have 
$$\||u|^2u\|_{X^{s+\sigma,b'-1}}\lesssim \|u\|_{X^{s,b}}^3.$$
\end{cor}
\begin{proof}
We describe the proof in dimension $d=2$, which is the hardest case. At the end of the proof we comment on the case $d=1$.
Using a dyadic decomposition and duality we need to estimate for $v\in X^{0,1-b'}$
$$M_0^{\sigma + s}\left|\int u_{M_1}\bar u_{M_2} u_{M_3}\bar v_{M_0} dx dt \right|.$$
Without loss of generality we can assume that $M_1\geq M_3$. Also, below we present the calculation as if  $v\in X^{0,b}$, but we can adjust this by possibly losing an $\varepsilon$ on the highest frequency, and this can be done by assuming that $b'<b$. We then consider two cases, when 
$M_1\geq  M_2$ or when $M_1\leq M_2$.  In the first case the most dangerous situation is when 
$M_0\sim M_1$. In this case we have 
$$M_0^{\sigma+s}\left|\int u_{M_1}\bar u_{M_2} u_{M_3}\bar v_{M_0} dxdt\right|\lesssim M_0^{\sigma+s}\|u_{M_1}u_{M_3}\|_{L^2}
\|\bar u_{M_2} \bar v_{M_0}\|_{L^2}.$$
After renormalizing and using the lemma above with $C_{+}$ we can continue with 
\begin{eqnarray*}&\lesssim& M_0^{\sigma+s} M_1^{-\frac12}M_3^{\frac12}M_0^{-\frac12}M_2^{\frac12}\|u_{M_1}\|_{X^{0,b}}\|u_{M_2}\|_{X^{0,b}}\|u_{M_3}\|_{X^{0,b}}\|v_{M_0}\|_{X^{0,b}}\\
&\lesssim& \label{R^2-Smoothing}
M_0^{\sigma-1}M_3^{\frac12-s}M_2^{\frac12-s}\|u_{M_1}\|_{X^{s,b}}\|u_{M_2}\|_{X^{s,b}}\|u_{M_3}\|_{X^{s,b}}\|v_{M_0}\|_{X^{0,b}}. 
\end{eqnarray*}
If $s\leq \frac12$ then we need $\sigma-2s<0$ and hence $\sigma<2s$. If $s> \frac12$ then we need
$\sigma<1$.
Assume  now that  $M_1\leq M_2$ and that again $M_0\sim M_2$. With a similar argument we  estimate 
$$M_0^{\sigma+s}\left|\int u_{M_1}\bar u_{M_2} u_{M_3}\bar v_{M_0} dxdt\right|\lesssim M_0^{\sigma+s}\|u_{M_1}\bar u_{M_2}\|_{L^2}
\|\bar v_{M_0}  u_{M_3}\|_{L^2}.$$
After renormalizing and using the lemma above with $C_{-}$ we obtain a similar  result.

The previous argument (and Lemma \ref{bilinear}) adapts to the case $d=1$; see also \cite{MR1357398}. Alternatively one can deduce them from  
the case $d=2$ using 
Lemmata 3.1 and 3.6 in \cite{MR1854113}.

\end{proof}

\section{Deterministic Results} 

\subsection{The Linear Schr\"odinger Equation on $\mathbb{T}^d$}\label{Deterministic ResultsLinear}

\

 In this section we focus on maximal estimates of the linear Sch\"odinger flow
\begin{equation}\label{MaxInTorus}
\left\| \sup_{0 \leq t \leq 1} |e^{it\Delta } f(x) | \right\|_{L^2_x(\mathbb{T}^d)} \lesssim \| f(x) \|_{H^{s}_x(\mathbb{T}^d)} \, .
\end{equation}
As mentioned in Section 2, it is standard that this estimate implies pointwise convergence $e^{it\Delta } f(x) \to f(x)$ as $t \to 0$ for almost every
$x \in \mathbb{T}^d$. The problem of identifying the minimal regularity $s$ for which   
\eqref{MaxInTorus} holds is still open. The following result has been proved in \cite{MR2409184} when $d=1$ and in \cite{WangZhang} when $d\geq2$. The proof is based on Strichartz estimates. 
However, comparing with \cite{WangZhang}, the exponent in the next proposition is better for $d \geq 3$ due to the use of the optimal periodic Strichartz estimates 
from \cite{BourgainDemeter}. Thus, we recall the proof for the sake of completeness.

\begin{proposition}\label{MaxInTorusPartial}
The inequality \eqref{MaxInTorus} holds for all $s > \frac{d}{d+2}$.  
\end{proposition}

\begin{proof}
By dyadic frequency decomposition (here $N \in 2^{\mathbb N}$) it is sufficient to prove that  
\begin{equation}\label{MaxInTorusEquiv}
\left\| \sup_{0 \leq t \leq 1} |e^{it\Delta } \operatorname{P}_N f |\right\|_{L^2(\T^d)} \lesssim N^{s} \| \operatorname{P}_N f \|_{L^{2}(\T^d)} \, 
\end{equation}
holds for all $ s > \frac{d}{d+2}$;
we recall that $\operatorname{P}_N$ is the frequency projection into the annulus of size $N$.

We actually prove the stronger estimate 
\begin{equation}\label{MaxInTorusEquivStrong}
\left\| \sup_{0 \leq t \leq 1} |e^{it\Delta } \operatorname{P}_N f | \right\|_{L^{2\frac{d+2}{d}}(\T^d)} \lesssim N^{\frac{d}{d+2}+} \| \operatorname{P}_N f \|_{L^{2}(\T^d)} \, 
\end{equation}
We use the following inequality (see \cite{MR2264734}), that holds by the Fundamental Theorem of Calculus and H\"older's inequality, 
\begin{equation}\label{LeeIneq}
\sup_{0 \leq t \leq 1} |\phi(t)| \lesssim |\phi(0)| + \alpha^{\frac{1}{p} - 1} \| \partial_t \phi(t) \|_{L^{p}_{t}([0, 1])} 
+ \alpha^{\frac{1}{p}} \|  \phi(t) \|_{L^{p}_{t}([0, 1])}\, , 
\end{equation}
with $\phi(t) = e^{it\Delta } \operatorname{P}_N f(x)$. The parameter $\alpha > 0$ will be chosen later in such a way as to equalize
the second and third term on the right hand side of 
\eqref{LeeIneq}.
Since $\partial_t e^{it\Delta} \operatorname{P}_N f(x) = i \Delta e^{it\Delta} \operatorname{P}_N f(x)$ we obtain 
\begin{align}\label{HWCONVTORUS}
& \sup_{0 \leq t \leq 1} |e^{it\Delta } \operatorname{P}_N f(x)|
\\ \nonumber
&  \lesssim |\operatorname{P}_N f(x)| + \alpha^{\frac{1}{p} - 1} \|   \Delta e^{it\Delta}  \operatorname{P}_N f(x) \|_{L^{p}_{t}([0,1])} 
+ \alpha^{\frac{1}{p}} \|  e^{it\Delta} \operatorname{P}_N f(x) \|_{L^{p}_{t}([0,1])}   
\\ \nonumber
&  \lesssim |\operatorname{P}_N f(x)| + \alpha^{\frac{1}{p} - 1} N^2  \|   e^{it\Delta}  \operatorname{P}_N f(x) \|_{L^{p}_{t}([0,1])} 
+ \alpha^{\frac{1}{p}}  \|  e^{it\Delta} \operatorname{P}_N f(x) \|_{L^{p}_{t}([0,1])}.   
\end{align}  
Specifying $\alpha = N^{2}$ and taking the $L^{p}_x(\T^d)$ norm of \eqref{HWCONVTORUS} we arrive at
\begin{align}
&\left\| \sup_{0 \leq t \leq 1} |e^{it\Delta } \operatorname{P}_N f| \right\|_{L^{p}(\T^d)}
 \lesssim \| \operatorname{P}_N f \|_{L^{p}(\T^d)} +
N^{\frac{2}{p}}   \|   e^{it\Delta}  \operatorname{P}_N f \|_{L^{p}( \T^d \times [0,1]  )}
\\ \nonumber
& \hspace{1in} \lesssim N^{\frac{d}{2} - \frac{d}{p}} \| \operatorname{P}_N f \|_{L^{2}(\T^d)} +
N^{\frac{2}{p}}   \|   e^{it\Delta}  \operatorname{P}_N f \|_{L^{p}( \T^d \times [0,1] )}
 \, ,
\end{align}   
where in the second estimate we used Bernstein's inequality.
Letting $p = 2\left( \frac{d+2}{d}\right)$ and using the Strichartz estimates \eqref{FullStrich} we obtain \eqref{MaxInTorusEquivStrong}.

\end{proof}
In the proof of Proposition \ref{MaxInTorusPartial} one can replace $\T^d$ by $\R^d$. However, in the latter 
case more powerful techniques are available.

We now  analyze the sharpness of the maximal estimate \eqref{MaxInTorus}. We exhibit a counterexample which shows that in fact 
the maximal estimate \eqref{MaxInTorus} fails for sufficiently rough data. Proposition \ref{couter} below  
follows adapting the non periodic counterexamples to the periodic setting. 
Remarkably, it has been proved in \cite{MR2409184} that in the case $d=1$ there are different counterexamples which lead to the 
necessity of $s > 1/4$, exploiting the quadratic Gauss summation. Here we give another one dimensional counterexample for the necessity of $s > 1/4$, 
based on the Galilean 
invariance of the Schr\"odinger equation. It is worth to mention that the quadratic Gauss summation approach is instead related 
to the pseudoconformal invariance. The 
equivalence between these symmetries in the convergence problem has been 
already observed when the (linear) problem was settled on $\R^d$, comparing 
the counterexmples in \cite{1608.07640, MR3613507} and \cite{Bourgain2016, MR3903115}. 

\begin{proposition}\label{couter}
The inequality \eqref{MaxInTorus} fails for all $s < \frac{d}{2(d+1)}$.
 \end{proposition}

\begin{proof}
One can adapt the non periodic counterexamples. The maximal inequality \eqref{MaxInTorus} is disproved using a family
of initial data frequency supported on a ball of radius $R > 1$ (letting then $R \to \infty$). The time interval at which the $R$-th member $f_R$
of this family attains the 
$\sup_{t} |e^{it \Delta} f_R(x)|$ is contained in $[0,1/R]$, for all $R >1$. Since in this time/frequency region there is essentially no
difference between periodic and non periodic solutions, the procedure is straightforward.  
 If $d = 1$ one can prove this statement in different ways (see \cite{MR2409184}). Here we also give an alternative proof. 
Let $\kappa \in (0, \frac{1}{2})$, $N \gg 1$ and 
$D=\lfloor N^{1-\kappa} \rfloor$. We first focus on the 
family of initial data
\begin{equation}\label{UnmodInitialData}
f (x)= \sum_{\substack{ k \in \Z \\ |k| \leq  N/D}}  e^{i D k  x}
\end{equation}
and the corresponding solutions
\begin{equation}\label{OurSol}
e^{it\Delta} f (x) = \sum_{|k| \leq N/D} e^{i(  D k  x - D^2 |k|^2 t)}  \, .
\end{equation}
Notice that 
\begin{equation}\label{fegfkdhsd}
|e^{it\Delta} f (x)|  \sim \sum_{|k| \leq N/D} 1  \sim N/D \sim 
N^{\kappa }
\qquad 
\mbox{if}
\qquad
 (x, t) \in X \times T
\, ,
\end{equation}
where 
$$
X  := D^{-1} \mathbb{Z} + B \left(0, \frac{1}{10 N}\right) , 
\qquad
T := D^{-2} \mathbb{Z} \, .
$$
This is because we can write elements $t \in T$ as $\lfloor N^{1-\kappa} \rfloor^{-2} \tau$ with $\tau \in \mathbb{Z}$, so that
$$
D^2 |k|^2 t = D^2 |k|^2  D^{-2} \tau = |k|^2 \tau \in \mathbb{Z}
$$
and
we can write elements $x \in X$ as $x = D^{-1} \ell + \varepsilon$, with $\ell \in \mathbb{Z}$, $\varepsilon \in \R$ with $|\varepsilon| \leq \frac{1}{10N}$,  
so that 
$$
D k  x = D k   ( D^{-1} \ell + \varepsilon) 
=  k    \ell + D k   \varepsilon \in \mathbb{Z} + B \left(0,  D  |k|  |\varepsilon | \right)
$$
and
$$
D  |k|  |\varepsilon | \leq N^{1-\kappa} N^{\kappa} \frac{1}{10 N} \leq \frac{1}{10} \, .
$$
We consider instead the modulated initial data
$$
\tilde f (x) = e^{i x} f(x)  \, ,
$$
where $f$ is chosen as in \eqref{UnmodInitialData}. The corresponding solutions are
$$
e^{it\Delta} \tilde f (x) = e^{i(x   -  t )} (e^{it\Delta} f)(x - 2 t )   \, .
$$
Thus, recalling  
\eqref{fegfkdhsd}, we have
\begin{equation}\label{fegfkdhsdBis} 
\sup_{0 \leq t \leq 1/D} |e^{it\Delta} \tilde f (x) | \gtrsim 
N^{\kappa }
\qquad 
\mbox{ if }
x \in \bigcup_{t \in T \cap [0, 1/D]} X +  t \, ,
\end{equation}
Since the set $\bigcup_{t \in T \cap [0, 1/D]} X +  t$ is equidistributed in $\mathbb{T}$,
its measure is of order  
$1 \wedge ( D^2 N^{-1}) \sim 1 \wedge N^{1- 2 \kappa})$,
 the first factor being the cardinality of $T \cap [0, 1/D]$, the second the cardinality of the small balls of~$X \cap \mathbb{T}$, and the last the volumes of these
balls. Thus the set~$\bigcup_{t \in T \cap [0, 1/D]} X +  t$ has full measure for all large $N$ since we assume~$\kappa < 1/2$.
By~\eqref{fegfkdhsdBis} and noting that 
$$\| \tilde f \|_{L^2(\T)} = \|  f \|_{L^2(\T)} \sim (N/D)^{1/2} \sim N^{\frac{\kappa}{2}} \, ,$$ 
the 
maximal inequality~\eqref{MaxInTorus} implies that
$$
N^{\kappa} \lesssim N^s N^{\frac{\kappa}{2}} \, .
$$
Letting $N \to \infty$ this leads to a contradiction if $s < \frac{\kappa}{2}$. Since we have restricted to~$\kappa < \frac{1}{2}$
we have disproved the inequality \eqref{MaxInTorus} for all $s < \frac{1}{4}$.

\end{proof}

\subsection{The NLS Equation on $\T^d$ and $\R^d$ (Theorem \ref{MainTHM1})}\label{Sec:ProofOfTheorem1}

\

In this section we prove Theorem \ref{MainTHM1}. Thus we focus on the NLS equation \eqref{NLSeq} on $\Omega^d$, where $\Omega = \R$
or $\Omega=\T$. The nonlinearity is $\mathcal{N}(z) = \pm |z|^{p-1}z$ with $p \geq 3$. We focus on initial data in $H^{s}(\Omega^d)$ 
with 
\begin{equation}\label{SubriticalExp}
s > \max \left( 0, \frac{d}{2} - \frac{2}{p-1} \right) \, .
\end{equation}
For such data the flow is locally well defined; see \cite{MR1383498, tao2006nonlinear} for the non-periodic case and \cite{MR1209299} for the periodic one.  
Let $\Phi^N_t $ be the flow associated to the truncated NLS equation 
\begin{equation}\label{TruncatedNLS}
i \partial_t  \Phi^N_t f + \Delta \Phi^N_t f = \operatorname{P}_{\leq N}  \mathcal{N} (\Phi^N_t f  )\, ,
\end{equation}
with initial datum $\Phi^N_0 f  := \operatorname{P}_{\leq N} f$.
As usual $\operatorname{P}_{\leq N}$ denotes the frequency projection on the ball of radius $N$ centered in the origin.
We write $\Phi_t f := \Phi^\infty_t  f$ for the flow of the NLS equation with initial datum $f = \operatorname{P}_{\infty} f$.
We also denote $\operatorname{P}_{>N} := \operatorname{P}_{\infty} - \operatorname{P}_{\leq N}$ and as alreday mentioned $\operatorname{P}_{N} := \operatorname{P}_{\leq N} - \operatorname{P}_{\leq N/2}$. 

The following 
maximal estimate ensures a.e. pointwise convergence to the data. 
This is the nonlinear analog of the maximal estimate \eqref{OmegaMaxEst}.

\begin{proposition}\label{MaxEstObv}
Let $f \in L^2(\Omega^d)$ be such that 
\begin{equation}\label{MaxInNonlinearFull}
\lim_{N \to \infty} \| \sup_{0 \leq t \leq \delta} | \Phi_t f(x) - \Phi^N_t f(x)   | \|_{L^2_x(B_1)} = 0 
\end{equation}
for any $B_1 \subset \Omega^d$. Then
$\Phi_t f(x) \to f(x) $ as $t \to 0$ for almost every $x \in \Omega^d$. 
\end{proposition}

From the proof it will be clear that in \eqref{MaxInNonlinearFull} we can replace the $L^2$ norm with the (smaller) $L^1$ norm. 
However is usually convenient to work in $L^2$ setting.

\begin{proof}
To prove Proposition \ref{MaxEstObv} we
decompose the difference as follows:

\begin{align}\label{DecompLikeThis}
| \Phi_t f(x) - f(x) | &
 \leq 
| \Phi_t f(x) - \Phi^N_t f(x) |
+
| \Phi^N_t f(x) - \operatorname{P}_{\leq N} f(x) |
+  
| \operatorname{P}_{> N} f(x) |
\end{align}
and pass to the limit $t \to 0$. The second term on the right hand side is zero. In fact, since 
$\operatorname{P}_{\leq N} f$ is smooth once has immediately that
$$
\lim_{t \to 0}  \Phi^N_t f(x) =  \operatorname{P}_{\leq N} f(x) \, ,
$$  
for all $x \in \Omega^d$.
So we arrive at\footnote{Hereafter we remove the $x$ variable in the argument of decompositions like \eqref{DecompLikeThis} to simplify the notation.}
$$
\limsup_{t \to 0} |\Phi_t f - f | \leq 
\limsup_{t \to 0} | \Phi_t f - \Phi^N_t f |
+  
| \operatorname{P}_{> N} f | \, .
$$
Let $\lambda >0$. Using the Chebyshev inequality 
\begin{align}\nonumber
| \{ x \in B_1 : &\limsup_{t \to 0} | \Phi_t f - f | > \lambda \} |
 \leq | \{ x \in B_1 : \sup_{0 \leq t \leq \delta} | \Phi_t f - \Phi^N_t f | > \lambda/2 \} |
\\ \nonumber
& \hspace{2.5in}
+
| \{ x \in B_1 : | \operatorname{P}_{> N} f  | > \lambda/2 \} |
\\ \nonumber
&
\hspace{1in} \lesssim
  \lambda^{-2} \left( \left\| \sup_{0 \leq t \leq \delta} | \Phi_t f - \Phi^N_t f | \right\|^2_{L^2(B_1)} 
+
  \| \operatorname{P}_{> N} f  \|^2_{L^2(B_1)} \right) \, ,
\end{align}
where $|\cdot|$ is the Lebesgue measure.
On the other hand we have 
 $ \|\operatorname{P}_{> N} f    \|_{L^2(\Omega^d)}  
 \to 0 $ as $N \to \infty$ (since $f \in L^{2}(\Omega^d)$)
and
$$
\lim_{N \to \infty} \left\| \sup_{0 \leq t \leq \delta} | \Phi_t f - \Phi^N_t f | \right\|_{L^2(B_1)} = 0 
$$
by assumption \eqref{MaxInNonlinearFull}. Thus we arrive to
$$
| \{ x \in B_1 : \limsup_{t \to 0} | \Phi_t f - f | > \lambda \} | 
= 0 
$$ 
and the statement follows taking the union over $\lambda>0$ and covering $\Omega^d$ with a countable collection of balls $B_1$.

\end{proof}

We combine the following lemma with the embedding contained in Lemma \ref{MainLemma} to verify the maximal 
estimate hypothesis of Proposition \ref{MaxEstObv} in concrete situations.  

\begin{lemma}\label{LocalLemma}
Let $p \geq 3$ and  $s > \max ( 0, \frac{d}{2} - \frac{2}{p-1} )$. Then 
\begin{equation}\label{MainTechEst1}
\|   \mathcal{N}  (u) -  \mathcal{N} (v)    \|_{X^{s, - \frac12 + + }}
\lesssim \left( \| u \|^{p-1}_{X^{s, \frac12 + }} + \| v \|^{p-1}_{X^{s, \frac12 + }} \right) \| u - v \|_{X^{s, \frac12 + }}
\end{equation}
\end{lemma}
We postpone the proof of Lemma \ref{LocalLemma} to the end of the section. 

We denote $R_0 = \| f \|_{H^s(\Omega^d)}$. Recall that $\eta$ is a smooth cut--off of $[0,1]$. Taking $\delta = \delta(R_0) <1$ sufficiently small
and combining \eqref{Basic1}, \eqref{Basic3}, \eqref{Basic2} and Lemma \ref{LocalLemma} 
one can show that the map
\begin{equation}\label{Contraction1}
\Gamma (u(x,t)) = \eta(t) e^{it \Delta} \operatorname{P}_{\leq N} f(x) - i \eta(t) \int_{0}^{t} e^{i(t-t')\Delta} \operatorname{P}_{\leq N} \mathcal N (u(x,t')) dt'
\end{equation}
is a contraction on the ball $\{ u \; : \; \| u \|_{X^{s, \frac12+}_{\delta}}\leq 2R_0\}$, for all~$N \in 2^{\mathbb{N}} \cup \{ \infty \}$. 
This is a standard argument, so we omit the proof (see for instance \cite[Section 3.5.1]{ETBook}). Moreover, a similar computation is part of the proof of Theorem \ref{MainTHM1}.   
However, we stress that the value of $\delta$ is uniform
in~$N \in 2^{\mathbb{N}} \cup \{ \infty \}$. In particular we have 
\begin{equation}\label{NUniformBound}
\| \Phi^N_t f \|_{X^{s, \frac12+}_{\delta}} \leq 2 R_0, 
\qquad \mbox{for all } \, N \in 2^{\mathbb{N}} \cup \{ \infty \} \, .
\end{equation}
We are now ready to prove Theorem \ref{MainTHM1}.
\subsubsection*{Proof of Theorem \ref{MainTHM1}}   
We first prove the a.e. convergence statement \eqref{AlmostEverywhereConv}.
By Lemma \ref{MainLemma} we have 
  \begin{equation}\nonumber
 \left\| \sup_{0 \leq t \leq \delta} | \Phi_t f(x) - \Phi^N_t f(x)   | \right\|_{L^2_x(B_1)}  \lesssim 
 \|  \Phi_t f - \Phi^N_t f    \|_{X^{s,\frac12+}_{\delta}} \, .
 \end{equation} 
Thus using Proposition \ref{MaxEstObv} it suffices to show that the right hand side goes to zero as $N \to \infty$. 
For $t \in [0, \delta]$ we have (see \eqref{Contraction1})
\begin{align}\nonumber
& \Phi_t f(x) - \Phi^N_t f(x) 
\\ \nonumber
& = \eta(t) e^{it \Delta} \operatorname{P}_{>N} f(x) - i \eta(t) \int_{0}^{t} e^{i(t-t')\Delta} \left( \mathcal N (\Phi_{t'} f(x)) - \operatorname{P}_{\leq N} \mathcal N (\Phi^N_{t'} f(x)) \right) dt'.
\end{align}
Then using \eqref{Basic1} and  \eqref{Basic3} we have 
\begin{align}
 \|  \Phi_t f - \Phi^N_t f    \|_{X^{s,\frac12+}_{\delta}}
\lesssim
\|  \operatorname{P}_{>N} f   \|_{H^{s}(\Omega^d)}
 + 
\|  \mathcal N (\Phi_t f) - \operatorname{P}_{\leq N} \mathcal N (\Phi^N_t f)     \|_{X^{s, -\frac12+}_{\delta}} \, .
\end{align}
To handle the nonlinear contribution we further decompose
$$
\mathcal N (\Phi_t f) - \operatorname{P}_{\leq N} \mathcal N (\Phi^N_t f )
= \operatorname{P}_{\leq N}
 \left( \mathcal N (\Phi_t f) -  \mathcal N (\Phi^N_t f ) \right)
 +
\operatorname{P}_{>N} \mathcal N (\Phi_t f) \, 
$$
so that 
\begin{align}\label{Plug1B}
 \|  \Phi_t f - \Phi^N_t f    \|_{X^{s,\frac12+}_{\delta}}
 & \lesssim
\|  \operatorname{P}_{>N} f   \|_{H^{s}(\Omega^d)}
 +
 \| \operatorname{P}_{>N}  \mathcal N (\Phi_t f)      \|_{X^{s, -\frac12+}_{\delta}} 
\\ \nonumber
& \qquad + 
 \| \operatorname{P}_{\leq N} \left(  \mathcal N (\Phi_t f) - \mathcal N (\Phi^N_t f)  \right)   \|_{X^{s, -\frac12+}_{\delta}} \, .
\end{align}
Then by \eqref{Basic2},
Lemma \ref{LocalLemma},
and \eqref{NUniformBound}, we get
\begin{align}\label{Plug1A}
  &  \|  \operatorname{P}_{\leq N} \left( \mathcal N (\Phi_t f) -  \mathcal N (\Phi^N_t f ) \right)  \|_{X^{s, -\frac12+}_{\delta}} 
  \lesssim \delta^{0+} R_{0}^{p-1}
  \|  \Phi_{t} f -  \Phi^N_{t} f    \|_{X^{s,  \frac12+}_{\delta}} \, ,
\end{align}
where we recall $R_{0} = \| f \|_{H^s(\Omega^d)}$.
Plugging \eqref{Plug1A} into \eqref{Plug1B}, taking $\delta=\delta(R_0)$ small enough and absorbing 
$$
 \delta^{0+} R_{0}^{p-1}
  \|  \Phi_{t} f -  \Phi^N_{t} f    \|_{X^{s,  \frac12+}_{\delta}}  \leq \frac12  
  \|  \Phi_{t} f -  \Phi^N_{t} f    \|_{X^{s,  \frac12+}_{\delta}}
  $$ 
into the left hand side, we arrive to 
\begin{align}\label{TRHSOTA}
 \|  \Phi_t f - \Phi^N_t f    \|_{X^{s,\frac12+}_{\delta}}
&  \lesssim
\|  \operatorname{P}_{>N} f   \|_{H^{s}(\Omega^d)}
 +
 \| \operatorname{P}_{>N}  \mathcal N (\Phi_{t} f)      \|_{X^{s, -\frac12+}_{\delta}} 
\end{align}
The  right hand side of \eqref{TRHSOTA} goes to zero as $N \to \infty$
since $f \in H^{s}(\Omega^d)$ and 
$\mathcal N (\Phi_{t} f) \in X^{s, -\frac12+}_{\delta}$; in fact applying Lemma \ref{LocalLemma} with $v=0$ 
and recalling~\eqref{NUniformBound} we have 
\begin{equation}\nonumber
\|  \mathcal N (\Phi_{t} f)      \|_{X^{s, -\frac12+}_{\delta}}  \lesssim
\| \Phi_{t} f      \|^p_{X^{s, \frac12+}_{\delta}} \lesssim 
 R_0^p \, .
\end{equation}
This concludes the proof of \eqref{AlmostEverywhereConv}.

To prove \eqref{EverywhereConv} it is enough to show that if $d=1,2$ and $ s>d/6$ then 
\begin{equation}\label{ProofOfEverywhereConv}
\left\|\int^t_0 e^{i(t-t')\Delta}|\Phi_{t'}f |^2 \Phi_{t'}f dt'\right\|_{X^{\frac{d}{2}+,\frac12+}_{\delta}} \lesssim \| \Phi_{t}f \|_{X^{s,\frac12++}_{\delta}}^3
\lesssim R_0^3 \, .
\end{equation}
Indeed then we would have  
$\Phi_{t}f - e^{it\Delta} f \in X^{\frac{d}{2}+,\frac12+}_{\delta}$ and we can use 
$X^{\frac{d}{2}+,\frac12+}_{\delta} \hookrightarrow C_t ([0,\delta] ;  H^{\frac{d}2+}(\Omega^d))$ and $H^{\frac{d}{2}+}(\Omega^d) \hookrightarrow C_x(\Omega^d)$ 
(Sobolev embedding) 
to get \eqref{EverywhereConv}.
On the other hand \eqref{ProofOfEverywhereConv} follows by \eqref{Basic3}, Corollary \ref{uuu} and \eqref{NUniformBound}, so we are done.

\hfill $\Box$

We conclude this  section with the proof of  Lemma \ref{LocalLemma} and the statement of a similar one -- an analog for functions with frequencies restricted to dyadic annuli.
These kind of results are now very well understood; however we report the proof for the sake of completeness.
\subsubsection*{Proof of Lemma \ref{LocalLemma}}
We consider the case $\Omega = \T$. The proof in the case $\Omega = \R$ requires some modification. It is in fact easier, since there is no loss 
in the endpoint case $p=2\left(\frac{d+2}{s}\right)$ of the Strichartz estimates \eqref{FullStrich}. 
We abbreviate everywhere in the proof 
$$
L^{q}_{x,t}(\T^{d+1})
\quad \mbox{to} \quad 
L^{q}_{x,t} \, .
$$
Recalling $\mathcal N (z) := |z|^{p-1} z$ and using the Fundamental Theorem of Calculus we can represent 
\begin{align}
& \mathcal N  ( u ) -  \mathcal{N} (v) 
= \int_0^1 \frac{d}{d \rho} ( \mathcal N (v + \rho (u-v) ) ) d \rho
\\ \nonumber
& = (u-v) \int_0^1 (\partial_{z} \mathcal N)(v + \rho (u-v) ) d\rho + (\overline{u-v} ) \int_0^1 (\partial_{\bar z} \mathcal N)(v + \rho  (u-v) ) d \rho
\\ \nonumber
&
=: (u-v) \zeta_1(u,v) + (\overline{u-v} ) \zeta_2(u,v) \, .
\end{align}
Notice that $\partial_z \mathcal N$ and $\partial_{\bar z} \mathcal N$ are continuous functions since $p \geq 3$.
For simplicity we only show that the 
$X^{s, - \frac12 ++}$ norm of $(u-v) \zeta_1(u,v)$ is bounded by the right hand side of \eqref{MainTechEst1}. The proof that 
the same holds for $(\overline{u-v} ) \zeta_2(u,v)$ is identical. 
Let decompose dyadically      
\begin{equation}
\|  (u-v) \zeta_1(u,v) \|_{X^{s, - \frac12++ }}^2 = \sum_{N} N^{2s} \| \operatorname{P}_{N} ( (u-v) \zeta_1(u,v) ) \|_{X^{0, - \frac12++ }}^2  
\end{equation}
and estimate
\begin{align}\label{Bony}
 N^s \| \operatorname{P}_{N} ( (u-v) \zeta_1(u,v) ) \|_{X^{0, - \frac12++ }}
& \lesssim  N^s \| \operatorname{P}_{N} ( (u-v)  \operatorname{P}_{\ll N } \zeta_1(u,v)  ) \|_{X^{0, - \frac12++ }}
\\ \nonumber
& + 
\sum_{N_1 \gtrsim N} N^{s} \| \operatorname{P}_{N} ( (u-v) \operatorname{P}_{N_1} \zeta_1(u,v) )  \|_{X^{0, - \frac12++ }} \, .
\end{align}
We first focus on the second term  on the right hand side of \eqref{Bony}. This one is the easiest to bound, since the restriction on frequencies 
 $N_1 \gtrsim N$ gives a gain once we estimate 
the norm of $\operatorname{P}_{N_1}  \zeta_1(u,v) $.
Using \eqref{Dual1}, H\"older's inequality, and \eqref{ITFCOS2} we get 
\begin{align}\label{Plug0} 
& \sum_{N_1 \gtrsim N} N^{s} \| \operatorname{P}_{N} ( (u-v) \operatorname{P}_{N_1} \zeta_1(u,v) )  \|_{X^{0, - \frac12++ }} 
\\ \nonumber
& 
\lesssim \sum_{N_1 \gtrsim N} N^{s +}  \| (u-v) \operatorname{P}_{N_1} \zeta_1(u,v)   \|_{L^{2 \left(\frac{d+2}{d+4}\right)+}_{x, t}}
\\ \nonumber
&
\lesssim 
\sum_{N_1 \gtrsim N} N^{s +}  \| u-v   \|_{L^{2 \left(\frac{d+2}{d}\right)}_{x, t}}
\| \operatorname{P}_{N_1}  \zeta_1(u,v)   \|_{L^{ \frac{d+2}{2}+}_{x, t}}
\\ \nonumber
&
\lesssim 
\sum_{N_1 \gtrsim N} N^{0+}  \| u-v   \|_{ X^{s,  \frac12+  }}
\| \operatorname{P}_{N_1}  \zeta_1(u,v)   \|_{L^{ \frac{d+2}{2}+}_{x, t}}.
\end{align}

Recalling the 
definition of $\zeta_1(u,v)$ and using Minkowski's inequality and $L^p$ estimates for nonlinear operators of power type (see for instance \cite[Proposition 2.3]{MR3947642})
we have  
\begin{align}\label{TBecomes}
& \| \operatorname{P}_{N_1}  \zeta_1(u,v)   \|_{L^{ \frac{d+2}{2}+}_{x, t}} 
 \lesssim 
N_1^{0+} \left( \| u \|^{p-2}_{L^{2\left(\frac{d+2}{d-2s}\right)-}_{x,t}} + \| v \|^{p-2}_{L^{2\left(\frac{d+2}{d-2s}\right)-}_{x,t}} \right) 
\\ \nonumber
& \times
\sum_{N_2} 
\min \left( 1, \frac{N_2}{N_1}\right)
\left( \| \operatorname{P}_{N_2} u \|_{L^{ \frac{2(d+2)}{4 - (p-2)(d-2s)}+}_{x, t}} 
+ \| \operatorname{P}_{N_2} v\|_{L^{ \frac{2(d+2)}{4 - (p-2)(d-2s)}+}_{x, t}} \right) .
\end{align}
Notice that \eqref{ITFCOS2} gives
\begin{equation}\label{2Ways}
\| \operatorname{P}_{N_2} F(x,t) \|_{L^{  \frac{2(d+2)}{4 - (p-2)(d-2s)} +}_{x, t}} 
\lesssim      N_2^{ 0 --- } \| \operatorname{P}_{N_2} F(x, t) \|_{X^{s, \frac12+}} 
 \, .
\end{equation}
Indeed, if $ \frac{2(d+2)}{4 - (p-2)(d-2s)} \geq \frac{2(d+2)}{d}$ we can use \eqref{ITFCOS2} and we get a factor
$N_2^{ \frac{(p-1)(d-2s) -4}{2}+}$. Since $\frac{(p-1)(d-2s) -4}{2}+ < 0$ for $s >\frac{d}{2} - \frac{2}{p-1}$ we get \eqref{2Ways}.
If $ \frac{2(d+2)}{4 - (p-2)(d-2s)} \leq \frac{2(d+2)}{d}$ we can bound the 
$L^{ \frac{2(d+2)}{4 - (p-2)(d-2s)}}_{x,t}$ norm with the  $L^{\frac{2(d+2)}{d}}_{x,t}$ norm and use \eqref{ITFCOS2}, to get a factor
$N_2^{- s + }$. Again, since $s > 0 $ we get \eqref{2Ways}.
Using \eqref{2Ways} and 
\eqref{ITFCOS2Summed} the estimate \eqref{TBecomes} becomes
\begin{align}\label{Plug1}
 \| \operatorname{P}_{N_1}  \zeta_1(u,v) &  \|_{L^{ \frac{d+2}{2}+}_{x, t}} 
 \lesssim 
N_1^{0+} \left( \| u \|^{p-2}_{ X^{s, \frac12+} } + \| v \|^{p-2}_{ X^{s, \frac12+ } } \right) 
\\ \nonumber
& \times
\sum_{N_2} 
\min \left( 1, \frac{N_2}{N_1}\right) N_2^{ 0 -- } 
\left( \| \operatorname{P}_{N_2} u \|_{X^{s, \frac12+}} + \| \operatorname{P}_{N_2} v \|_{X^{s, \frac12+}} \right) 
\\ \nonumber 
& 
\lesssim
N_1^{0--}  
\left( \|  u \|^{p-1}_{X^{s, \frac12+}} + \|  v \|^{p-1}_{X^{s, \frac12+}} \right) 
\end{align}
Plugging \eqref{Plug1} into \eqref{Plug0} 
we obtain (recall $N_1 \gtrsim N$)
\begin{align}\nonumber
& \sum_{N_1 \gtrsim N} N^{s} \| \operatorname{P}_{N} ( (u-v) \operatorname{P}_{N_1} \zeta_1(u,v) )  \|_{X^{0, - \frac12++ }} 
\\ \nonumber
&
\lesssim 
\sum_{N_1 \gtrsim N}  N_1^{0-}    \| u-v   \|_{ X^{s,  \frac12+  }}
\left( \|  u \|^{p-1}_{X^{s, \frac12+}} + \|  v \|^{p-1}_{X^{s, \frac12+}} \right) 
\\ \nonumber
&
\lesssim
N^{0-}  \| u-v   \|_{ X^{s,  \frac12+  }}
\left( \|  u \|^{p-1}_{X^{s, \frac12+}} + \|  v \|^{p-1}_{X^{s, \frac12+}} \right) \, .
\end{align}
Summing the square of this inequality over $N$, we have handled the contribution of the second term on the 
right hand side of~\eqref{Bony}. To handle the first term  we note that
\begin{align}\label{TLPIDet}
N^s \| \operatorname{P}_{N} ( (u-v)  \operatorname{P}_{\ll N } \zeta_1(u,v)  ) \|_{X^{0, - \frac12++ }}
\lesssim
N^s \|  (\operatorname{P}_{\sim N} (u-v) )  \operatorname{P}_{\ll N } \zeta_1(u,v)   \|_{X^{0, - \frac12++ }} 
\end{align}
and decompose   
\begin{align}\label{WSTCubeGain}
(\operatorname{P}_{\sim N} (u-v) )  \operatorname{P}_{\ll N } \zeta_1(u,v) 
& =
\sum_{N_1 \ll N} \sum_{Q_{N, N_1}} ( \operatorname{P}_{Q_{N, N_1}} (u-v)  ) \operatorname{P}_{N_1} \zeta_1(u,v) 
\\ &
=
\sum_{N_1 \ll N} \sum_{Q_{N, N_1}} \operatorname{P}_{100 Q_{N, N_1}} \left( ( \operatorname{P}_{Q_{N, N_1}} (u-v)  ) \operatorname{P}_{N_1} \zeta_1(u,v) \right)
\, ,
\end{align}
where $Q_{N, N_1}$ is a partition of the annulus 
of size $N$ into cubes of side $N_1$ (this is possible since $N_1 < N$).  In the second identity we used that 
the support of $( \operatorname{P}_{Q_{N, N_1}} F  ) \operatorname{P}_{N_1} G$ is contained in 
$100 Q_{N, N_1}$. Since for different $N, N_1$ the projections $\operatorname{P}_{Q_{N, N_1}}$ are (almost) orthogonal, 
squaring  \eqref{WSTCubeGain}
we get
 \begin{align}\label{CubeLocGainPreq}
 N^{2s} \| &\operatorname{P}_{N} ( (u-v)  \operatorname{P}_{\ll N } \zeta_1(u,v)  ) \|_{X^{0, - \frac12++ }}^2
\\ \nonumber
& 
= 
 N^{2s} \sum_{N_1 \ll N} 
 \sum_{Q_{N, N_1}} \| \operatorname{P}_{100 Q_{N, N_1}} \left( ( \operatorname{P}_{Q_{N, N_1}} (u-v)  ) \operatorname{P}_{N_1} \zeta_1(u,v) \right)
\|_{X^{0, - \frac12++ }}^2 \, .
 \end{align}
Proceeding exactly as before we get 
\begin{align}\label{CubeLocGain}
  \| \operatorname{P}_{100 Q_{N, N_1}}  & \left( ( \operatorname{P}_{Q_{N, N_1}} (u-v)  ) \operatorname{P}_{N_1} \zeta_1(u,v) \right)
   \|_{X^{0, - \frac12++ }} 
\\ \nonumber 
& \lesssim
N_1^{0-}  \| \operatorname{P}_{Q_{N, N_1}} ( u-v )   \|_{ X^{0,  \frac12+  }}
\left( \|  u \|^{p-1}_{X^{s, \frac12+}} + \|  v \|^{p-1}_{X^{s, \frac12+}} \right) \, ;
\end{align}
notice that since the side of $Q_{N, N_1}$ is $N_1$ we had only powers of $N_1$ in this computation.
Thus 
\begin{align}\label{CubeLocGain2}
 N^{2s} \| \operatorname{P}_{N} & ( (u-v)  \operatorname{P}_{\ll N } \zeta_1(u,v)  ) \|_{X^{0, - \frac12++ }}^2
\lesssim 
\\ \nonumber
&
\left( \|  u \|^{p-1}_{X^{s, \frac12+}} + \|  v \|^{p-1}_{X^{s, \frac12+}} \right)^2 
N^{2s} \sum_{N_1 \ll N} 
 \sum_{Q_{N, N_1}} 
 N_1^{0-}  \| \operatorname{P}_{Q_{N, N_1}} ( u-v )   \|^2_{ X^{0,  \frac12+  }}
 \, .
 \end{align}
Summing the square of \eqref{CubeLocGain2} over $Q_{N, N_1}$ (recall that these cubes are a partition of the annulus of size $N$) 
and later over $N_1$, we obtain (after taking the square root) 
$$
N^{s} \| \operatorname{P}_{N} ( (u-v)  \operatorname{P}_{\ll N } \zeta_1(u,v)  ) \|_{X^{0, - \frac12++ }}
\lesssim
\left( \|  u \|^{p-1}_{X^{s, \frac12+}} + \|  v \|^{p-1}_{X^{s, \frac12+}} \right)
\| \operatorname{P}_{N} ( u-v )   \|_{ X^{s,  \frac12+  }}
  \, ,
$$
which gives the correct control also on the first term  on the right hand side of~\eqref{Bony}. This concludes the proof.

\hfill $\Box$

Later we will also need the following Lemma, whose proof is a straightforward adaptation of the previous argument.
\begin{lemma}
Let $d=2$ and $s >0$.   
Let $M_1 \geq M_2 \geq M_3$
 be dyadic scales. Then
 \begin{multline}\label{TrilinearStrich}
 \| (\operatorname{P}_{M_1}F) (\operatorname{P}_{M_2} G) (\operatorname{P}_{M_3} H)   \|_{X^{s, -\frac12++}}  \\ \lesssim 
  \| \operatorname{P}_{M_1} F \|_{X^{s, \frac12+}}  \| \operatorname{P}_{M_2} G  \|_{X^{0+, \frac12+}} \| \operatorname{P}_{M_3} H  \|_{X^{0, \frac12+}}.
\end{multline}
\end{lemma}

\section{Probabilistic Results}

\subsection{The Linear Schr\"odinger Equation on $\T^d$ with Random Data}\label{linSchrTd}

\
Here we prove almost surely uniform convergence of the randomized Schr\"odinger flow to the initial datum, at the $H^{0+}$ level. More precisely, we show that 
$e^{it \Delta} f^{\omega} \to f^{\omega}$ as $t \to 0$ uniformly over $x \in \T^d$ and   
$\mathbb{P}$-almost surely for data $f^\omega$ defined as
\begin{equation}\label{RandInitData}
f^\omega(x) = \sum_{n \in \mathbb{Z}^d} \frac{g_n^\omega}{\lb n \rb^{\frac{d}{2}+ \alpha}} e^{i n \cdot x}, \qquad x\in \T^d \, ,
\end{equation} 
where $\alpha >0$ and
each $g_n^\omega$ is complex and independently drawn from a standard normal distribution. In fact, the argument we present works 
for independent $g_n^\omega$ drawn from any distribution with sufficiently strong decay properties. We present the standard normal case for definiteness. 
Fix $t \in \R$. We have that $\mathbb{P}$-almost surely   
\[ e^{it \Delta} f^\omega \in \bigcap_{s <\alpha} H^s(\T^d). \] 
This is an immediate consequence of \eqref{HyperInitialdataPreq} below, taking the union over~$\varepsilon >0$.
Moreover, for all $t \in \R$ the $e^{it\Delta}f^{\omega}$ are $\mathbb{P}$-almost surely continuous functions\footnote{In fact they 
belong to $\bigcap_{s <\alpha} C^s(\T^d)$ $\mathbb{P}$-almost surely, 
but we will never need this stronger information.} of the $x$ variable. 
This is a consequence of the higher integrability property \eqref{LInftyHigherIntegrability} below, from which one can easily deduce 
uniform convergence as $N \to \infty$
of the sequence~$P_{\leq N} f^{\omega}$, with probability larger than~$1-\varepsilon$. So the limit~$f^{\omega}$ is continuous 
with the same probability, and the almost sure continuity follows taking the union over~$\varepsilon >0$.

Now we prove the first part of Theorem \ref{MainTHM2}, namely

\begin{proposition}\label{RandomLinCOnv1}
Let $\alpha >0$. 
For $\mathbb{P}$-almost every $f^\omega$ of the form \eqref{ReprInitialData} we
have that  
$$
e^{it\Delta} f^\omega(x) \to f^\omega(x) \quad \mbox{as $t \to 0$} 
$$
for every $x \in \T^d$ and uniformly. 
\end{proposition}

This proposition proves the first part of Theorem \ref{MainTHM2}. Its proof  appears at the end of this section after we establish few lemmata.

We start recalling the following well--known concentration bound:
\begin{lemma}[{\cite[Lemma 3.1]{BurqTzvet}}] \label{largedev}There exists a constant $C$ such that 
\begin{equation}\label{LargDevB} 
\left\| \sum_{n \in \Z^d} g_n^\omega \; a_n \right\|_{L^r_\omega} \leq C r^\frac12 \| a_n\|_{\ell^2_n(\Z^d)}
\end{equation}
for all $r \geq 2$ and $\{ a_n \} \in \ell^2(\Z^d)$. 
\end{lemma}
Using \eqref{LargDevB} with $a_n = e^{i n \cdot x - i |n|^2 t} \lb n \rb^{-\frac{d}{2} - \alpha}$ we obtain  for $r \geq 2$ that
for $f^\omega$ an in~\eqref{RandInitData}
\begin{equation}
 \| \operatorname{P}_N\! e^{it \Delta} f^\omega \|_{L^r_\omega} \leq C r^{\frac12} N^{-\alpha} \, , 
 \end{equation}
with a constant uniform in $t \in \R$. 
From this, we also have improved $L^p_x$ estimates for randomized data.

\begin{lemma}\label{LDBMink}
Let $p \in [2,\infty)$. Assume $f^\omega$ is as in \eqref{RandInitData}. There exists constants $C$ and $c$, 
independent of $t \in \R$, 
such that
\begin{equation}\label{SubGaussianTail} 
\mathbb{P}(\| \operatorname{P}_N\! e^{it \Delta} f^\omega\|_{L^p_x(\T^d)} > \lambda) \leq Ce^{-cN^{2\alpha}\lambda^2} \, .
\end{equation}
In particular, for any $\varepsilon >0$ sufficiently small, we have 
\begin{equation}\label{HigherIntegrability}
 \| \operatorname{P}_N\! e^{it \Delta} f^\omega \|_{L^p_x(\T^d)} \lesssim N^{-\alpha} \left( - \ln \varepsilon \right)^{1/2},
 \quad N \in 2^{\Z} \cup \{ \infty \} \, , 
 \end{equation}
with probability at least $1 - \varepsilon$. 
Thus 
\begin{equation}\label{LInftyHigherIntegrability}
 \| \operatorname{P}_N\! e^{it \Delta} f^\omega \|_{L^{\infty}_x(\T^d)} \lesssim N^{-\alpha +} \left( - \ln \varepsilon \right)^{1/2} ,
 \quad N \in 2^{\Z} \cup \{ \infty \} \, ,
\end{equation}
with probability at least $1 - \varepsilon$.
\end{lemma}
\begin{proof}
We prove \eqref{HigherIntegrability}, then \eqref{LInftyHigherIntegrability} follows by Bernstein inequality.
By Minkowski's inequality and Lemma \ref{largedev} above, we have for any $r \geq p\geq2$
\begin{align*}
\left( \int  \|\operatorname{P}_N\! e^{it \Delta} f^\omega\|_{L^p_x(\T^d)}^r d \mathbb{P}(\omega) \right)^\frac1r 
 \leq \Bigl\|  \| \operatorname{P}_N\! e^{it \Delta} f^\omega\|_{L^r_\omega} \Bigr\|_{L^p_x(\T^d)} 
\leq C N^{-\alpha}r^\frac12.
\end{align*}
which is enough to conclude that $\|\operatorname{P}_N\! e^{it \Delta} f^\omega\|_{L^p_x(\T^d)}$ is a sub-Gaussian random variable
satisfying the tail bound \eqref{SubGaussianTail}.
%

\end{proof}
%
Proceeding as in the proof of Lemma \ref{LDBMink} we also obtain
improved Strichartz estimates for randomized data. 
\begin{lemma} 
Let $p \in [2,\infty)$. Assume $f^\omega$ is as in \eqref{RandInitData}. Then we have, for some constants $C$ and~$c$, independent of 
$t \in \R$ 
the bound
\[ \mathbb{P}\left( \| e^{it\Delta} \operatorname{P}_N\! f^\omega \|_{L^p_{x,t}(\T^{d+1})} > \lambda \right) \leq C e^{-cN^{2\alpha}\lambda^2}.\]
In particular, for any $\varepsilon >0$ sufficiently small, we have 
\begin{equation}\label{ImprStrIneq}
 \| e^{it\Delta} \operatorname{P}_N\! f^\omega \|_{L^p_{x,t}(\T^{d+1})} \lesssim N^{-\alpha} \left( - \ln \varepsilon \right)^{1/2} ,
  \quad N \in 2^{\Z} \cup \{ \infty \} \, ,
 \end{equation}
with probability at least $1 - \varepsilon$. Thus
\begin{equation}\label{LInftyImprStrIneq}
 \| e^{it\Delta} \operatorname{P}_N\! f^\omega \|_{L^\infty_{x,t}(\T^{d+1})} \lesssim N^{-\alpha +} \left( - \ln \varepsilon \right)^{1/2} , 
 \quad N \in 2^{\Z} \cup \{ \infty \} \, ,
 \end{equation}
with probability at least $1 - \varepsilon$. 
\end{lemma}

Fix $t \in \R$. Later we will also need the following bound (with high probability) for the $H^{s}$ norm of $e^{it \Delta}f^{\omega}$
with $s < \alpha$. This is a well know fact that we recall applying again 
\eqref{LargDevB} with $a_n = e^{i n \cdot x - |n|^2 t} \lb n \rb^{-\frac{d}{2} - \alpha +s}$, so that we get for $r \geq 2$ 
\[ 
\|  \operatorname{P}_N \langle D \rangle^s \! e^{it\Delta} f^\omega \|_{L^r_\omega} \leq C r^{\frac12} N^{s-\alpha}, 
\qquad s <\alpha \, .
\]
Here $\lb D \rb$ denotes the Fourier multiplier operator $\lb n \rb$. 
Proceeding as in the proof of Lemma \ref{LDBMink} we also obtain
\[
 \mathbb{P}\left( \|  \langle D \rangle^s \operatorname{P}_N\! e^{it\Delta} f^\omega \|_{L^2_{x}(\T^d)} > \lambda \right) \leq C e^{-cN^{2(\alpha-s)}\lambda^2}, 
\qquad s <\alpha 
\]
and in particular, for any $\varepsilon >0$ sufficiently small 
\begin{equation}\label{HyperInitialdataPreq}
 \|  e^{it \Delta} f^\omega \|_{H^s_{x}(\T^d)} \lesssim  \left( - \ln \varepsilon \right)^{1/2} 
 \qquad  s < \alpha, \quad t \in \R
 \, , 
 \end{equation}
with probability at least $1 - \varepsilon$. Again the constant is uniform on $t \in \R$.

We thanks Chenjie Fan for sharing with use the argument we used in the next proof,  and that strengthens   our original a.e. convergence to a uniform one.

\subsubsection*{Proof of Proposition \ref{RandomLinCOnv1}}

Let us decompose 
\begin{equation}\label{PlugInUniform}
| e^{it\Delta} f^\omega - f^\omega | \leq 
| e^{it\Delta} \operatorname{P}_{>N} f^\omega |
+
| e^{it\Delta} \operatorname{P}_{\leq N} f^\omega - \operatorname{P}_{\leq N} f^\omega |
+  
| \operatorname{P}_{>N} f^\omega |.
\end{equation}
We fix $\lambda >0$ and $\varepsilon > 0$ sufficiently small. Using \eqref{LInftyImprStrIneq} we see that 
\begin{equation}\label{PlugInUniform1}
\| e^{it\Delta} \operatorname{P}_{>N} f^\omega \|_{L^{\infty}_{x, t}(\T^{d+1})} + \| \operatorname{P}_{>N} f^\omega \|_{L^{\infty}_{x}(\T^d)} <  \lambda/2
\end{equation}
holds
for all $N$ sufficiently large, depending on $\lambda, \varepsilon$, with probability larger than $1 - \varepsilon$. 
Let us fix such $N^* = N^* (\lambda, \varepsilon)$. We also fix $s^* >  \frac{d}{2}$. Since
$$
e^{it\Delta} \operatorname{P}_{\leq N^*} f^\omega - \operatorname{P}_{\leq N*} f^\omega
= \sum_{|n| \leq N^*} (e^{-it |n|^2} - 1) e^{in \cdot x} \hat{f^{\omega}}(n) , \, , 
$$
using Cauchy--Schwartz, the summability of $\langle n \rangle^{-2s^*}$ and \eqref{HyperInitialdataPreq} with $s=0$, $t=0$ (in the last inequality) we get
\begin{align}\nonumber
\| e^{it\Delta} \operatorname{P}_{\leq N^*} f^\omega - \operatorname{P}_{\leq N*} f^\omega \|_{L^{\infty}_{\T^d}}
&  
\lesssim \sup_{|n| \leq N^{*} } | e^{-it |n|^2} - 1  |  \left( \sum_{|n| \leq N^{*}} \langle n \rangle^{2s^*} |\hat{f^{\omega}}(n)|^2 \right)^{1/2} 
\\ 
&
\lesssim  |t| (N^{*})^{s^*+2}  \| f^{\omega} \|_{L^2}  \lesssim   |t|(N^{*})^{s^* + 2} (- \ln \varepsilon)^{1/2}  \, ,
\end{align} 
with probability larger than $1 - \varepsilon$. 
Thus
we can find $t^*$ sufficiently small, depending only on $N^{*}$ and $\varepsilon$, such that 
for all $t \in (0, t^*)$ we have
\begin{equation}\label{PlugInUniform2}
\| e^{it\Delta} \operatorname{P}_{\leq N^*} f^\omega - \operatorname{P}_{\leq N^*} f^\omega \|_{L^{\infty}_x (\T^d)} < \lambda / 2
\end{equation}
with probability larger than $1 - \varepsilon$. 
Plugging \eqref{PlugInUniform1}, \eqref{PlugInUniform2} into \eqref{PlugInUniform} we see the following. Given $\lambda > 0$, we 
have found $t^* = t^* (\lambda, \varepsilon)$ such that 
\begin{equation}
\| e^{it\Delta} f^\omega - f^\omega \|_{L^{\infty}_x(\T^d)} < \lambda, \qquad \mbox{for all $t \in (0, t^*)$} \, ,
\end{equation}
with probability larger than $1 - \varepsilon$.
Namely for all $\varepsilon > 0$ we have 
uniform (in $x$) convergence 
$e^{it\Delta} f^\omega \to f^{\omega}$ as $t \to 0$ for all $\omega \in A_{\delta}$ with $\mathbb{P}(A_{\delta}) > 1-2\varepsilon$. 
This means that the uniform convergence fails only if $\omega$ belongs to any of the complementary sets
$A_{\varepsilon}^C$, thus in particular for $\omega \in \bigcap_{k \in \N} A_{1/k}$. Since 
Since $\mathbb{P}\left(  \bigcap_{k \in \N} A_{1/k} \right) = \lim_{k \to 0} \mathbb{P}(A_{1/k}) = 0$ the statement follows. 

\hfill $\Box$

\subsection{The Linear Schr\"odinger Equation on $\R^d$ with Random Data}

\label{SectionlinRandRd}

\

For the linear Schr\"odinger equation on $\R^d$, randomization arguments
similar to those in Section \ref{linSchrTd} can be applied. 
Given $s >0$ and $f \in H^{s}(\R^d)$, we work with the randomized data $f^{\omega}$ defined in \eqref{eq:RdRand}.
As in the periodic case, this argument works for any independent random variables  whose distribution functions 
decay sufficiently rapidly. We work with the standard normal distribution for the sake of definiteness.

 By arguments almost identical to those for the periodic case, we have for any $p \in [2,\infty)$
\begin{align*}
\| P_N e^{it\Delta} f^\omega\|_{L^p(\R^d \times [0,1])} \lesssim N^{-s} (- \ln \varepsilon)^{1/2} \| P_N f \|_{L^2(\R^d)} , 
 \quad N \in 2^{\Z} \cup \{ \infty \} \, ,
\end{align*}
for $\omega$ in a set of probability at least $1-\varepsilon$.
Thus the Bernstein inequality gives  
\begin{align}
\| P_N e^{it\Delta} f^\omega\|_{L^{\infty}(\R^d \times [0,1])} \lesssim N^{-s +} (- \ln \varepsilon)^{1/2} \| P_N f \|_{L^2(\R^d)} ,
 \quad N \in 2^{\Z} \cup \{ \infty \} \, , \label{ASCOBern}
\end{align}
for $\omega$ in a set of probability at least $1-\varepsilon$. In particular
$e^{it \Delta} f^{\omega}$ is $\mathbb{P}$-almost surely continuous\footnote{In fact one can show it is $\mathbb{P}$-almost surely in $C^{s'}_x$ for all $s' < s$.}. 

Moreover one has uniform in $t \in \R$ bounds for 
the $H^s$ norm
\begin{equation}\label{HsRandomizNormOnR}
\| e^{it \Delta} f^\omega \|_{H^s(\R^d)} \lesssim (-\ln \varepsilon)^{1/2} \| f \|_{H^s(\R^d)} \, ,
\end{equation}
for $\omega$ in a set of probability at least $1-\varepsilon$. 
For more general versions of these estimates we refer to \cite[Lemmata 2.1 \& 2.3]{OhOkamPoco}.

Using these estimates and proceeding exactly as in the periodic case, we can establish the first part of Theorem \ref{MainTHM3}, namely
\begin{proposition}\label{RandomLinCOnv2}
Let $s >0$ and $f\in H^s(\R^d)$. For $\mathbb{P}$-almost every $f^\omega$ of the form \eqref{eq:RdRand} we
have  
$$
e^{it\Delta} f^\omega(x) \to f^\omega(x) \quad \mbox{as $t \to 0$} 
$$
for every $x \in \R^d$ and uniformly. 
\end{proposition}


\subsection{The Cubic NLS Equation on $\T^d$ ($d=1,2$) with Random Data (Theorem \ref{MainTHM2})}\label{WickOrderedNLS}

\

In this section, we consider the cubic Wick-ordered NLS \eqref{eq:wickNLS} on $\T^d$ ($d=1,2$) as in the work of  
 Bourgain in \cite{bourgain1996invariant}. 
Namely, we look at the nonlinearity 
$$
\mathcal{N}(u) := \pm u \left( |u|^2 - 2 \mu \right), 
\qquad 
\mu := \fint_{\T^d} |u(x, t)|^2 dx \, .
$$
We are interested again in randomized initial data, i.e. $f^{\omega}$ is taken to be of the form \eqref{RandInitData}. Recall (see \eqref{HyperInitialdataPreq}) that such data 
is $\mathbb{P}$-almost surely 
in $H^s$ for all $s < \alpha$  
and
\begin{equation}\label{HyperInitialdata}
\| f^\omega  \|_{H^{s}} \lesssim \left( - \ln \varepsilon \right)^{1/2}, \quad
s < \alpha \, ,
\end{equation}
with probability at least $1 - \varepsilon$, for all $\varepsilon \in (0,1)$ sufficiently small. Since we work with any $\alpha >0$, we 
are considering initial data in $H^{0+}$.
We approximate equation \eqref{eq:wickNLS} as in \eqref{TruncatedNLS}, for all $N \in 2^{\mathbb{N}} \cup \{ \infty \}$. Recall that
$\Phi^{N}_t f^{\omega}$ denotes the associated flow, with initial datum
$$
\Phi^{N}_0 f^{\omega} :=  \operatorname{P}_{\leq N} f^{\omega} =  \sum_{|n| \leq N} \frac{g_n^\omega}{\lb n \rb^{\frac{d}2+ \alpha}} e^{i n \cdot x}  \, .
$$ 
We write $\Phi_t f^{\omega} = \Phi^{\infty}_t f^{\omega}$ for the flow of \eqref{eq:wickNLS} with datum $f^{\omega} = \operatorname{P}_{\infty} f^{\omega}$.   

\begin{proposition}\label{Bourgain1/2}
Let $d=1,2$ and $\alpha >0$. 
Let $N \in 2^{\mathbb{N}} \cup \{ \infty \}$. For all $\sigma \in [0, \frac12)$, the following holds. 
Assume 
\begin{equation}\label{BourgainAssumption}
u = u(I) + u(II), 
\quad u(I) = e^{it \Delta} \operatorname{P}_{\leq N} f^{\omega},
\quad \| u(II) \|_{X^{\alpha + \sigma, \frac12+}}   < 1 
\end{equation}
and the same for $v$.
Then 
\begin{equation}\label{WickThesis1}
\|   \mathcal{N}  (u)   \|_{X^{\alpha + \sigma, - \frac12+}} \lesssim 
\left( - \ln \varepsilon \right)^{3/2}
\end{equation}
\begin{equation}\label{WickThesis2}
\|   \mathcal{N}  (u) -  \mathcal{N} (v)    \|_{X^{\alpha + \sigma, - \frac12++}} \lesssim 
 \left( - \ln \varepsilon \right) 
\| u - v \|_{X^{\alpha + \sigma,  \frac12+}}  
\end{equation}
for initial data of the form \eqref{RandInitData}, with probability at least $1 - \varepsilon$, for all $\varepsilon \in (0,1)$ sufficiently small. 
If we take $u$ as in \eqref{BourgainAssumption} and we instead assume  
$$
v = v(I) + u(II), 
\quad v(I) = e^{it \Delta}  f^{\omega},
\quad \| u(II) \|_{X^{\alpha + \sigma, \frac12+}}   < 1 \, ,
$$ 
we have 
\begin{equation}\label{WickThesis3}
\|   \mathcal{N}  (u) -  \mathcal{N} (v)    \|_{X^{\alpha + \sigma, - \frac12++}} \lesssim N^{-\alpha} \, .
\end{equation}
\end{proposition}

\begin{remark}
Recall that $\alpha$ indicates  the regularity of the initial datum. We are denoting by $\sigma$ the amount of smoothing we can prove for 
the Wick--ordered cubic nonlinearity $\mathcal N$. 
More precisely, since the initial data \eqref{RandInitData} belongs to $H^{\alpha -}$, we can interpret this statement as saying that, with 
arbitrarily large probability, $\mathcal N$ is~$\sigma +$ smoother than $f^\omega$. Since $\sigma < \frac12$ is permissible, we reach 
$\frac12 -$ smoothing for $\mathcal N$ and, combining with \eqref{Basic3}, also for the Duhamel contribution 
$\Phi_{t}^N f^\omega - e^{it\Delta} P_{\leq N}f^\omega$.   
\end{remark}
We postpone the proof of Proposition \ref{Bourgain1/2} to the end of the section. 
Recall that $\eta$ is a smooth cut-off of the unit interval. 
Let us fix $\alpha > 0$. Using \eqref{Basic3}, \eqref{Basic2} and Proposition \ref{Bourgain1/2} one can show that
for all 
$\delta >0$ sufficiently small the following holds.
For all $N \in 2^{\mathbb{N}} \cup \{ \infty \}$, the map 
\begin{equation}\label{FixedPointDrift}
\Gamma^N(u) :=  \eta(t) e^{it\Delta} \operatorname{P}_{\leq N} f^{\omega} - i \eta(t) \int_0^t e^{i(t-s) \Delta} \operatorname{P}_{\leq N} \mathcal{N} (u (\cdot, s)) \, ds  
\end{equation}
is a contraction on the set
\begin{equation}\label{SetUnifBoundOutsideExceptSet}
\left\{ e^{it \Delta} \operatorname{P}_{\leq N} f^{\omega} + g, \quad \| g \|_{X^{\alpha + \sigma, \frac12+}_{\delta}}  < 1  \right\} 
\end{equation}
equipped with the $X^{\alpha + \sigma, \frac12+}_{\delta}$ norm, outside an exceptional set (we call it a $\delta$--exceptional set) of initial data of 
probability smaller than $e^{-\delta^{-\gamma}}$, with $\gamma >0$
a given small constant. Notice that this holds uniformly over $N \in 2^{\mathbb{N}} \cup \{ \infty \}$.
Again, this is a standard routine calculation that we omit (see for instance \cite[Section 3.5.1]{ETBook}). 
We only explain how to find the relation between the local existence time $\delta$ and the size of the exceptional set. 
Given any $\varepsilon \in (0,1)$ sufficiently small,  
using \eqref{Basic3}, \eqref{Basic2} and Proposition \ref{Bourgain1/2}, we have 
$$
\| \Gamma^N(u) -  \eta(t) e^{it\Delta} \operatorname{P}_{\leq N} f^{\omega} \|_{X^{\alpha + \sigma, \frac12+}_{\delta} }
\lesssim \delta^{0+} \left( - \ln \varepsilon \right)^{3/2} \, ,
$$
for all $f^{\omega}$ outside an exceptional set of probability smaller than $\varepsilon$. Letting $\delta$ such that $\varepsilon = e^{-\delta^{-\gamma}}$  
with $\gamma >0$ a fixed small constant, we have $C \delta^{0+} \left( - \ln \varepsilon \right)^{3/2} < 1$ for all $\delta >0$ sufficiently small.
Note that the measure $e^{-\delta^{-\gamma}}$ of the $\delta$--exceptional set converges to zero as $\delta \to 0$. 
In particular, for $\omega$ outside the $\delta$--exceptional set, the fixed point $\Phi^N_t f^{\omega}$ of the map \eqref{FixedPointDrift} belongs to 
the set \eqref{SetUnifBoundOutsideExceptSet}, namely   
\begin{equation}\label{UnifBoundOutsideExceptSet}
\| \Phi^N_t f^{\omega}
-  e^{it\Delta} \operatorname{P}_{\leq N} f^{\omega} \|_{X^{\alpha + \sigma, \frac12+}_{\delta} }  < 1, 
\qquad N \in 2^{\mathbb{N}} \cup \{ \infty \} \, .
\end{equation}

We are now ready to prove Theorem \ref{MainTHM2}. 
\subsubsection*{Proof of Theorem \ref{MainTHM2}}
Notice that \eqref{r2prob-lin}  is the content of Proposition \ref{RandomLinCOnv1}. To prove \eqref{AlmostEverywhereConvergenceRandom}, 
let us assume that we have proved 
\begin{equation}\label{NonlinearProbMaxEst}
\lim_{N \to \infty} \left\| \sup_{0 \leq t \leq \delta} | \Phi_{t} f^{\omega} (x) - \Phi^{N}_{t} f^{\omega} (x) | \right\|_{L^{2}_x(\T^2)} = 0
\end{equation}  
for all $f^{\omega}$ outside a $\delta$--exceptional set $A_\delta$. This means that given 
$f^{\omega}$ we can find, $\mathbb{P}$-almost surely, a $\delta_{\omega}$ such that \eqref{NonlinearProbMaxEst} is satisfied. Indeed, if we could not do so, this would mean
that $f^{\omega} \in \bigcap_{\delta >0} A_{\delta}$, and the probability of this event is zero, since $\mathbb{P}(A_\delta) \to 0$ as $\delta \to 0$. So, using
Proposition \ref{MaxEstObv} with $\delta = \delta_{\omega}$, we have $\mathbb{P}$-almost surely 
$$
\lim_{t \to 0} \Phi_{t}^{\omega} f^{\omega}(x) - f^{\omega}(x) = 0, 
\qquad \mbox{for a.e. $x \in \T^2$} \, ,
$$ 
as claimed. It remains to prove \eqref{NonlinearProbMaxEst}. We decompose
$$
|\Phi_t f^{\omega} - \Phi^N_t f^{\omega} |
\leq 
|e^{it \Delta} \operatorname{P}_{>N} f^{\omega}| + 
|\Phi_t f^{\omega} - e^{it \Delta}  f^{\omega}   -  (  \Phi^N_t f^{\omega} -  e^{it \Delta}  \operatorname{P}_{\leq N} f^{\omega}   ) | \, ,
$$
Thus, recalling the decay of the high frequency linear term given by \eqref{LInftyImprStrIneq}, it remains to show that
\begin{equation}
\lim_{N \to \infty} \left\| \sup_{0 \leq t \leq \delta}  |\Phi_t f^{\omega} - e^{it \Delta}  f^{\omega}  
 -  (  \Phi^N_t f^{\omega} -  e^{it \Delta}  \operatorname{P}_{\leq N} f^{\omega}   ) | \right\|_{L^{2}(\T^2)} = 0 \, ,
\end{equation}
for all $f^{\omega}$ outside a $\delta$--exceptional set. 

For any $\alpha >0$, we can choose $\sigma$ sufficiently close to $\frac12$ that 
\begin{equation}\label{CompareWithThisS}
s_{\T} < s_{\T^2} = \frac12 < \alpha + \sigma \, .
\end{equation}

Thus,
using the $X^{s,b}$ space embedding from Lemma \ref{MainLemma}, it suffices to prove
\begin{equation}\label{CubicPlugging1}
\lim_{N \to \infty} \left\| w - w^N  \right\|_{X^{\alpha + \sigma,\frac12+}_{\delta}} = 0 \, ,
\end{equation} 
where  
$$
w^N := \Phi^N_t f - e^{it\Delta} \operatorname{P}_{\leq N} f^{\omega}, 
\qquad
w := w^{\infty}   \, .
$$
Notice that by \eqref{UnifBoundOutsideExceptSet} we have
$$
\| w^N \|_{X^{\alpha + \sigma,\frac12+}_{\delta}} < 1, 
\qquad N \in  2^{\mathbb{N}} \cup \{ \infty \} \, .
$$
Since for $t \in [0, \delta]$ we have 
\begin{equation}
w -  w^N
= - i \eta(t) \int_{0}^{t'} e^{i(t-t')\Delta} \left( \mathcal N (\Phi_{t'} f^{\omega}) - \operatorname{P}_{\leq N} \mathcal N (\Phi^N_{t'} f^{\omega}) \right) dt' \, ,
\end{equation}
using \eqref{Basic3}, \eqref{Basic2}, we get 
\begin{equation}\label{CubicPlugging2}
\| w -  w^N  \|_{X^{\alpha + \sigma,\frac12+}_{\delta}}
\lesssim  \delta^{0+} \| \mathcal N (\Phi_{t} f) 
- \operatorname{P}_{\leq N} \mathcal N (\Phi^N_{t} f) \|_{X^{\alpha + \sigma, - \frac12++}_{\delta}} \, .
\end{equation}
%
%
%
We decompose
\begin{align}\label{CubicPlugging2Decomp}
\mathcal N (\Phi_{t} f) 
& 
- \operatorname{P}_{\leq N} \mathcal N (\Phi^N_{t} f) = 
\\ \nonumber
&
\operatorname{P}_{\leq N} \left( \mathcal N  (e^{it\Delta} \operatorname{P}_{\leq N} f^{\omega} + w) 
 -
\mathcal N (e^{it\Delta} \operatorname{P}_{\leq N} f^{\omega} + w^N) \right)
+
\mbox{Remainders} \, ,
\end{align}
where 
$$
\mbox{Remainders} := 
\operatorname{P}_{\leq N} \left( \mathcal N  (e^{it\Delta}  f^{\omega} + w) -
\mathcal N  (e^{it\Delta} \operatorname{P}_{\leq N} f^{\omega} + w) \right)
+
\operatorname{P}_{> N} \mathcal N  (\Phi_{t} f) \, .
$$
Notice that by \eqref{WickThesis1}, \eqref{WickThesis3} we have 
\begin{equation}\label{RemaindersToZero}
\|  \mbox{Remainders} \|_{X^{\alpha + \sigma, -\frac12++}_{\delta}} \to 0 \quad \mbox{as $N \to \infty$} \, ,
 \end{equation}
with probability at least $1 - \varepsilon$.
Using \eqref{WickThesis2} we can estimate
\begin{align}\label{CubicPlugging3}
 \| \operatorname{P}_{\leq N}  \left( \mathcal N  (e^{it\Delta} \operatorname{P}_{\leq N} f^{\omega} + w) 
\right.   
-  
\mathcal N (e^{it\Delta} \operatorname{P}_{\leq N} f^{\omega} + & \left. w_N) \right)    \|_{X^{\alpha + \sigma, -\frac12++}_{\delta}}
 \\ \nonumber
 & 
\lesssim 
\left( - \ln \varepsilon \right) \left\| w - w^N     \right\|_{X^{\alpha + \sigma,\frac12+}_{\delta}}
\end{align}
and \eqref{CubicPlugging2}, \eqref{CubicPlugging2Decomp}, \eqref{CubicPlugging3} give
\begin{equation}
 \left\|   w - w^N     \right\|_{X^{\alpha + \sigma,\frac12+}_{\delta}} 
 \lesssim
\delta^{0+} \left( - \ln \varepsilon \right)   \left\| w - w^N     \right\|_{ X^{\alpha + \sigma,\frac12+}_{\delta}} + 
\left\| \mbox{Remainders} 
  \right\|_{X^{\alpha + \sigma,-\frac12++}_{\delta}}
\end{equation}
with probability at least $1 - \varepsilon$.  
Since with our choice of $\varepsilon = e^{-\delta^{-\gamma}}$ we have $C \delta^{0+} \left( - \ln  \varepsilon \right)^{3/2} < 1$, 
we can absorb the first term on the right hand side into the left hand side and we still have that 
\eqref{RemaindersToZero} holds outside a $\delta$--exceptional set.
Thus letting $N \to \infty$ the proof of~\eqref{AlmostEverywhereConvergenceRandom} 
is complete.

To prove \eqref{EverywhereConvergenceRandom} we  
proceed as before. We show that for any $\delta >0$ sufficiently small we have 
\begin{equation}\label{LastThing}
\Phi_t f^{\omega} -  f^{\omega} \in X^{\frac{d}{2}+, \frac12 +}_{\delta} \, .
\end{equation}
for $f^{\omega}$ outside a $\delta$-exceptional set $A_{\delta}$. 
This means that given 
$f^{\omega}$ we can find, $\mathbb{P}$-almost surely, a $\delta_{\omega}$ such that \eqref{LastThing} is satisfied. Indeed, if we could not do so, this would mean
that $f^{\omega} \in \bigcap_{\delta >0} A_{\delta}$, and the probability of this event is zero, since $\mathbb{P}(A_\delta) \to 0$ as $\delta \to 0$.
Once we know that 
$\Phi_{t}f^{\omega} - e^{it\Delta} f^{\omega} \in X^{\frac{d}{2}+,\frac12+}_{\delta_{\omega}}$ 
we can use 
$X^{\frac{d}{2}+,\frac12+}_{\delta_{\omega}} \hookrightarrow C_t ([0,\delta_{\omega}] ;  H^{\frac{d}2+}(\Omega^d))$ and $H^{\frac{d}{2}+}(\Omega^d) \hookrightarrow C_x(\Omega^d)$ 
(Sobolev embedding) 
to get \eqref{EverywhereConvergenceRandom}, that so holds with probability $=1$.
To prove \eqref{LastThing} we use the smoothing, exactly as before, except that now we 
have to require the stronger inequality
$$
\frac{d}{2} < \alpha + \sigma \, .
$$  
Since we can take $\sigma <\frac12$ for $d = 1, 2$, we see that the previous condition is satisfied as long as $\alpha > \frac{d-1}{2}$.  
This concludes the proof.

\hfill $\Box$
 
 \begin{remark}
It is worthy to remark that, comparing with for instance \cite{bourgain1996invariant}, the procedure which allows to promote a statement valid on a $\delta$-exceptional set $A_\delta$
for arbitrarily small $\delta >0$ to a 
statement which is valid with probability $=1$ is far easier. In particular we only need that $\lim_{\delta \to 0} \mathbb{P}(A_{\delta}) =0$ 
but we do not need any efficient upper bound of the convergence rate. This is because we are considering a property which has to be verified only at 
time $t=0$ a.s.,
instead that in a given small time interval containing $t=0$, as in \cite{bourgain1996invariant}. 
 \end{remark}

We are now ready to prove the smoothing estimates given in Proposition \ref{Bourgain1/2}. 
\subsubsection*{Proof of Proposition \ref{Bourgain1/2}}
 Notice that the Wick--ordered nonlinearity can be written as
\begin{equation}\label{NewSumDiffIdentityPreq}
\mathcal{N}(u(x, \cdot))
= 
\sum_{n_2 \neq n_1, n_3} \widehat{u}(n_1) \widehat{\overline{u}}(n_2) \widehat{u}(n_3) e^{i (n_1 - n_2 + n_3) \cdot x}
- \sum_{n} \widehat{u}(n)  |\widehat{u}(n)|^2 e^{i n \cdot x} 
\end{equation}
where we are looking at the nonlinear term for fixed time and $\widehat{u}(\cdot)$ denotes the space Fourier coefficients.
From \eqref{NewSumDiffIdentityPreq}, exploiting the symmetry $n_1 \leftrightarrow n_3$, we also have the identity 
\begin{align}\label{NewSumDiffIdentity}
& \mathcal{N}(u(x, \cdot)) - \mathcal{N}(v(x, \cdot)) 
\\ \nonumber
& 
=
\sum_{n_2 \neq n_1, n_3} (\widehat{u}(n_1) - \widehat{v}(n_1)) \widehat{\overline{u}}(n_2) \widehat{u}(n_3) e^{i (n_1 - n_2 + n_3) \cdot x}
- \sum_{n} (\widehat{u}(n) - \widehat{v}(n)) |\widehat{u}(n)|^2 e^{i n \cdot x} 
\\ \nonumber
& 
+ \sum_{n_2 \neq n_1, n_3} (\widehat{u}(n_3) - \widehat{v}(n_3)) \widehat{\overline{v}}(n_2) \widehat{v}(n_1) e^{i (n_1 - n_2 + n_3) \cdot x}
- \sum_{n} (\widehat{u}(n) - \widehat{v}(n) )  |\widehat{v}(n)|^2 e^{i n \cdot x} 
\\ \nonumber
& 
+ \sum_{n_2 \neq n_1, n_3} ( \widehat{\overline{u}}(n_2) - \widehat{ \overline{v}}(n_2)) \widehat{v}(n_1) \widehat{v}(n_3) e^{i (n_1 - n_2 + n_3) \cdot x}
- \sum_{n} ( \widehat{\overline{u}}(n) - \widehat{\overline{v}}(n) )  \widehat{u}(n) \widehat{v}(n) e^{i n \cdot x}  \, .
\end{align} 
Using 
\eqref{NewSumDiffIdentity} (and recalling again the symmetry $n_1 \leftrightarrow n_3$), it is 
clear that we can reduce to proving 
the (more general) Lemma \ref{Bourgain1/2General} given below. It implies the desired statement since each summation in the above decomposition can be controlled by letting
$$
u_{j} (n_j) = u (n_j), \;\; v(n_j), \;\; \text{ or } \; u(n_j) - v(n_j) \, .
$$ 

\hfill$\Box$

\
The proof of Lemma \ref{Bourgain1/2General} below follows closely the arguments introduced by Bourgain in \cite{bourgain1996invariant}. We still display  the details
 since we need to quantify the gain of regularity. One will note though that the proof of Lemma \ref{Bourgain1/2General}  
 reported here is much easier than the one presented in \cite{bourgain1996invariant} since in our case $f^\omega$
is more regular, namely we consider $\alpha >0$ instead of $\alpha =0$.

\begin{lemma}\label{Bourgain1/2General}
Let $d=1,2$ and $\alpha >0$.  Let $N \in 2^{\mathbb{N}} \cup \{ \infty \}$. For all $\sigma \in [0, \frac12)$ the following holds. 
Assume for $j = 1, 2, 3$
\begin{equation}\label{DecU_j}
u_j(I) = e^{it \Delta} \operatorname{P}_{\leq N} f^{\omega},
\qquad
\| u_j(II) \|_{X^{\alpha + \sigma, \frac12+} } < 1  .
\end{equation}
Let $J_j \in \{  I, II \}$, $j=1, 2, 3$.
Then, for all
$\varepsilon \in (0,1)$ sufficiently small we have the following 
\begin{equation}\label{WWHTP1}
\|   \mathcal N  (u_1(J_1), \overline{u_2}(J_2), u_3(J_3))  \|_{X^{\alpha + \sigma, - \frac12+}} \lesssim 
  \left( - \ln \varepsilon \right)^{3/2} \, ,
\end{equation}
and more precisely
\begin{equation}\label{WWHTP2}
\|   \mathcal N  (u_1(II), \overline{u_2}(J_2), u_3(J_3))  \|_{X^{\alpha + \sigma, - \frac12++}} \lesssim 
 \left( - \ln \varepsilon \right)  \| u_1(II)  \|_{X^{\alpha + \sigma, \frac12+}}  \, ,
\end{equation}
\begin{equation}\label{WWHTP3}
\|   \mathcal N  (u_1(J_1), \overline{u_2}(II), u_3(J_3))  \|_{X^{\alpha + \sigma, - \frac12++}} \lesssim 
 \left( - \ln \varepsilon \right)  \| u_2(II)  \|_{X^{\alpha + \sigma, \frac12+}}  \, ,
\end{equation}
with probability at least $1 - \varepsilon$.
Moreover, if in \eqref{DecU_j} we replace for some $j=j^*$ the projection operator $\operatorname{P}_{\leq N}$
by $\operatorname{P}_{> N}$, then the estimate \eqref{WWHTP1} with $J_{j^*} = I$ holds with an extra factor $N^{-\alpha}$ on the right hand side. 
\end{lemma}
Notice that by the symmetry $n_1 \leftrightarrow n_3$ the estimate \eqref{WWHTP2} implies an analogous estimate for $u_3(II)$.
\

Before we pass to the proof we should remark that Lemma \ref{Bourgain1/2General} proves an almost sure gain of smoothness of $\sigma=\frac12-$  for the nonhomogeneous part of the solution of \eqref{eq:wickNLS} with initial data $f^\omega \in H^{\alpha-}, \, \alpha>0$.  This smoothing effect should be compared to the one recorded in Corollary \ref{uuu} proved in a deterministic manner. There we proved that if the  initial data is in $H^{0+}$ then basically there is only a $0++$ smoothing.

\begin{proof}
We prove \eqref{WWHTP1}, \eqref{WWHTP2}, \eqref{WWHTP3} in the case $N = \infty$.
It is then immediate to adapt the proof to $N \in \mathbb{N}$ and to prove the second part of the statement.
Moreover, we first give the proof in dimension $d=2$, which is the hardest case. At the end of the proof we explain how to handle 
the case $d=1$. 
We split the nonlinearity into two parts:
\begin{align*}
 \mathcal{N}_1  (u_1(J_1), \overline{u_2}(J_2), u_3(J_3)) 
 &= 
 \sum_{n_2 \neq n_1, n_3} \widehat{u_1(J_1)}(n_1) \widehat{\overline{u_2}(J_2)}(n_2) \widehat{u_3(J_3)}(n_3) e^{i (n_1 - n_2 + n_3) \cdot x} \, , \\
 \mathcal{N}_2  (u_1(J_1), \overline{u_2}(J_2), u_3(J_3)) 
 &= \quad \sum_{n} \quad \widehat{u_1(J_1)}(n) \widehat{\overline{u_2}(J_2)}(n) \widehat{u_3(J_2)}(n) e^{i n \cdot x} \, .
\end{align*}
We prove \eqref{WWHTP1}, \eqref{WWHTP2}, \eqref{WWHTP3} for $\mathcal N_{1}$, which is the most challenging contribution.   
The proof for $\mathcal{N}_2$ is elementary, so we leave the details to the reader. We decompose over dyadic scales 
$N_1, N_2 , N_3$ in the following way:
\begin{align}\nonumber
&  \|  \mathcal N_{1}  (u_1(J_1), \overline{u_2}(J_2), u_3(J_3))   \|_{X^{\alpha + \sigma, -\frac12++}}
 \\ \nonumber
 &
  \leq    \sum_{N_1, N_2 , N_3} 
  \| \mathcal N_{1}  (\operatorname{P}_{N_1} u_1(J_1), \operatorname{P}_{N_2} \overline{u_2}(J_2), 
  \operatorname{P}_{N_3} u_3(J_3))  \|_{X^{\alpha + \sigma, -\frac12++}}
 \\ \nonumber
& =:  \sum_{M_1, M_2, M_3} \| \mathcal N_{1}  (\operatorname{P}_{M_1} w_1(J_1), 
\operatorname{P}_{M_2} w_2(J_2), \operatorname{P}_{M_3} w_3(J_3) ) \|_{X^{\alpha + \sigma, -\frac12++}}
  \end{align}
where we denoted with $M_1, M_2, M_3$ the decreasing order of $N_1, N_2, N_3$. Notice that in this way  
$w_1$ denotes the $u_j$ supported on the largest frequency. 
We estimate this sum by first doing some reductions and then considering several cases. First we show that we 
can reduce to considering the case where the highest--frequency function is a random linear flow; i.e.
\begin{equation}\label{TheLargeIsRandom}
w_1( J_1) = w_1(I) \, .
\end{equation} 
Indeed if $w_1(J_1) = w_1(II)$ we get, using \eqref{TrilinearStrich}    
\begin{align}\label{SummingThis1}
&  \| \mathcal N_{1}  (\operatorname{P}_{M_1} w_1(II), \operatorname{P}_{M_2} w_2(J_2), \operatorname{P}_{M_3} w_3(J_3) ) \|_{X^{\alpha + \sigma, - \frac12++}}
\\ \nonumber
& \lesssim 
 \|  w_1(II) \|_{X^{\alpha + \sigma, \frac12+}}
 \|   w_2 (J_2) \|_{X^{0+, \frac12+}}  
\|   w_3 (J_3) \|_{X^{0, \frac12+}} 
 \, ,
\end{align}
On the other hand, recalling \eqref{DecU_j} and \eqref{HyperInitialdata} we have  
\begin{equation}\label{w_jBoundedness}
\| w_j (II) \|_{X^{\alpha + \sigma, \frac12+}} < 1, 
\qquad   
\| w_j (I) \|_{X^{0+, \frac12+}}   \lesssim \left( - \ln \varepsilon \right)^{1/2}, 
\end{equation}
where the second inequality holds with probability at least $1 - \varepsilon$. 
Thus, when $w_1( J_1) = w_1(II)$ the estimates \eqref{WWHTP1}, \eqref{WWHTP2}, \eqref{WWHTP3} follow 
summing the square of \eqref{SummingThis1} over $M_1, M_2, M_3$, factorizing the sum, and then using Plancherel and \eqref{w_jBoundedness}.   
To prove the second bound in \eqref{w_jBoundedness} 
one should notice that the space-time Fourier transform of $e^{it \Delta} f^\omega$ is
 $$
 \widehat{e^{it \Delta} f^\omega }(n, \tau) 
 = \frac{g^{\omega}}{\langle n \rangle^{\frac{d}{2} + \alpha} } \delta(\tau + |n|^2) \, ,
 $$ 
 where $\delta$ is the delta function. So a direct computation gives 
 $$
 \| e^{it \Delta} f^\omega  \|_{X^{0 + , \frac{1}{2} + } }^{2}
 = \sum_{n} \frac{|g_n^{\omega}|^2}{\langle n \rangle^{d + 2 \alpha -}}   < \infty \, ,
 $$ 
which using  
\begin{equation}\label{HypercontractivityFinal}
\int g_{n_j}^\omega g_{n_j'}^\omega d \mathbb{P}(\omega) = 0, 
\qquad
\int g_{n_j}^\omega \overline{g_{n_j'}^\omega} d \mathbb{P}(\omega) = \left\{ \begin{array}{lll} 0 & \mbox{if} & j \neq j' \\ 1 & \mbox{if} &  j= j' \end{array} \right. \, , 
\end{equation}
immediately implies  
$$
 \| \| e^{it \Delta} f^\omega  \|_{X^{0 + , \frac{1}{2} + } } \|_{L^{2}_{\omega}}^2 
 \sum_{n} \frac{1}{\langle n \rangle^{d + 2 \alpha -}}   < \infty  \, .
$$
Using the hypercontractivity (basically \eqref{HypercontractivityFinal} many times)
we can upgrade this to an $L^{p}_{\omega}$ bound, for any $p < \infty$, 
with a constant $Cp^{1/2}$ (see \cite[Proposition 4.5]{MR3518561} for details). Proceeding as in the proof of 
Lemma \ref{LDBMink}, this 
implies the second bound in \eqref{w_jBoundedness} 
for all $\omega$ outside a set of probability smaller than $\varepsilon$, as required.

Then we perform  a second reduction to remove frequencies which are far from the paraboloid. More precisely, we 
denote with 
$\operatorname{P}_{A}$ the space-time Fourier projection into the set $A$ and  
our goal is to reduce 
\begin{align}\label{SBBBoundALLTHECUTOFF}
& \sum_{M_1, M_2, M_3} \|  \mathcal N_{1} \left( \operatorname{P}_{M_1} w_1 (I) , 
\operatorname{P}_{M_2} w_2 (J_2) ,  \operatorname{P}_{M_3} w_3 (J_3) \right)  \|_{X^{\alpha + \sigma, -\frac12++}}^2
\\ \nonumber
&=
 \sum_{N, M_1, M_2, M_3} N^{2\alpha + 2\sigma} \| \operatorname{P}_N \mathcal N_{1} \left(  \operatorname{P}_{M_1} w_1 (I) , \operatorname{P}_{M_2} w_2 (J_2)  ,  \operatorname{P}_{M_3} w_3 (J_3)  \right) \|_{X^{0, -\frac12++}}^2
\end{align}
to 
\begin{equation}\label{FarFromParaboloid}
\sum_{N, M_1, M_2, M_3 } 
N^{2\alpha + 2\sigma } \|  \operatorname{P}_{N} \operatorname{P}_{\left\{ \lb \tau + |n|^2 \rb \leq N^{1 + \frac{1}{10}} \right\} } \mathcal N_{1} \left(  
  \operatorname{P}_{M_1} w_1 (I)  \operatorname{P}_{M_2} w_2 (J_2)   \operatorname{P}_{M_3} w_3 (J_3)   \right)  \|_{X^{0, -\frac12++}}^2 
 \end{equation}
(the $\frac{1}{10}$ is removable, however it does not create any problems and facilitates the computations).
To obtain this reduction, it is sufficient to show that projection of the nonlinearity onto the complementary set is appropriately bounded; i.e. that
 \begin{align}\label{FVLReduction}
& \sum_{N, M_1, M_2, M_3 } 
 N^{2\alpha + 2\sigma} \| \operatorname{P}_{N} \operatorname{P}_{\left\{ \lb \tau + |n|^2 \rb > N^{\frac{11}{10}} \right\} }
  \mathcal N_{1} \left(  \operatorname{P}_{M_1} w_1 (I) , \operatorname{P}_{M_2} w_2 (J_2) ,  \operatorname{P}_{M_3} w_3 (J_3)  \right) \|_{X^{0, -\frac12++}}^2
\\ \nonumber
& 
 \qquad \qquad \qquad \qquad
 \lesssim
    \left( - \ln \varepsilon \right)
\| w_2 (J_2) \|^2_{_{X^{0+, \frac12+}}}
\| w_3 (J_2) \|^2_{_{X^{0+, \frac12+}}}
 \end{align}
on a set of probability larger than $1- \varepsilon$. Indeed, recalling \eqref{w_jBoundedness} and summing over $N$, this would imply the validity of 
\eqref{WWHTP1}, \eqref{WWHTP2}, \eqref{WWHTP3} for this term.
We could have required a weaker bound than \eqref{FVLReduction}, replacing the $X^{0+, \frac12+}$ norm with an $X^{\alpha +\sigma, \frac12+}$ norm 
if $J_2 = II$ and with a $\ln \frac{1}{\varepsilon}$ factor if $J_2 = I$.  
However, we are able to prove the stronger  estimate \eqref{FVLReduction}. 
To do so we bound 
\begin{align}\label{PlugCaseEasy1}
&
 \sum_{ M_1, M_2, M_3 } N^{2\alpha + 2\sigma} \| \operatorname{P}_{N} \operatorname{P}_{\left\{ \lb \tau + |n|^2 \rb > N^{\frac{11}{10}} \right\} }  \mathcal N_{1} \left(  \operatorname{P}_{M_1} w_1 (I) , \operatorname{P}_{M_2} w_2 (J_2) ,  \operatorname{P}_{M_3} w_3 (J_3)  \right) \|_{X^{0, -\frac12++}}^2
\\ \nonumber
&
\sim
N^{2\alpha + 2\sigma}  \sum_{ \substack{ M_1, M_2, M_3 \\  n \sim N }  } \int 
\frac{ \chi_{ \{   \langle \tau + |n|^{2} \rangle > N^{\frac{11}{10}} \} } }{  \langle \tau + |n|^{2} \rangle^{1--}} 
\left| \widehat{\mathcal N_{1} (\cdot)} (n, \tau) \right|^2 
 \, d \tau
\\ \nonumber
&
\lesssim
N^{2 \alpha - }  \sum_{ \substack{ M_1, M_2, M_3 \\  n \sim N } } \int 
\left| \widehat{\mathcal N_{1} (\cdot)} (n, \tau) \right|^2 \, d \tau
\\ \nonumber
& \sim
 N^{2 \alpha - }  \sum_{  M_1, M_2, M_3  } 
\| \operatorname{P}_N 
\mathcal N_{1} \left(  \operatorname{P}_{M_1} w_1 (I) , \operatorname{P}_{M_2} w_2 (J_2) , \operatorname{P}_{M_3} w_3 (J_3)  \right) \|_{L^{2}_{x,t}}^2 \, .
\end{align}
Then using H\"older's inequality, the improved Strichartz inequality \eqref{ImprStrIneq} for randomized functions (for the $L^{q}$ norm of $w_1 (I)$),
and the Strichartz inequality~\eqref{ITFCOS2} (for the $L^{4}$ norms of $w_2(J_2)$ and  $w_3(J_3)$), we obtain
\begin{align}\label{PlugCaseEasy2}
\| \operatorname{P}_N 
& \mathcal N_{1} \left( 
  \operatorname{P}_{M_1} w_1 (I) , \operatorname{P}_{M_2} w_2 (J_2) , \operatorname{P}_{M_3} w_3 (J_3)  \right) \|_{L^{2}_{x,t}}^2
\\ \nonumber
& 
\leq \| \operatorname{P}_{M_1} w_1 (I) \|^2_{L^q_{x,t}} \| \operatorname{P}_{M_2} w_2 (J_2) \|^2_{L^{4+}_{x,t} } \| \operatorname{P}_{M_3} w_3 (J_3) \|^2_{L^{4+}_{x,t} }  \, .
\\ \nonumber
& 
\lesssim \left( - \ln \varepsilon \right) M_1^{-2\alpha} \| \operatorname{P}_{M_2} w_2 (J_2) \|^2_{L^{4+}_{x,t} } \| \operatorname{P}_{M_3} w_3 (J_3) \|^2_{L^{4+}_{x,t} }  \, ,
\\ \nonumber
&
\lesssim  \left( - \ln \varepsilon \right) M_1^{-2\alpha} 
 \| \operatorname{P}_{M_2} w_2 (J_2) \|^2_{X^{0+, \frac12+} } \| \operatorname{P}_{M_3} w_3 (J_3) \|^2_{X^{0+, \frac12+} } 
\end{align}
where we are taking $q \gg 1$ sufficiently large. 
This holds 
on a set of probability larger than $1 - \varepsilon$. Since $M_{1} \sim N$ 
once we plug \eqref{PlugCaseEasy2} into  
into \eqref{PlugCaseEasy1} the factor $N^{2\alpha-}$ is absorbed by $M_1^{-2\alpha}$ and we can rewrite the remaining factor as $M_{1}^{0-}N^{0-}$. 
Thus, summing over $N, M_1, M_2, M_3$ we obtain \eqref{FVLReduction}. So we have reduced to \eqref{FarFromParaboloid}. 

To handle this term we need a more explicit expression for the functions $w_j$. If we consider functions of the form $w(I)$ 
(here we omit the subscript $j$ to simplify the notation) we already know
\begin{equation}\label{BouLinRepr1}
w (I)(x,t) = \sum_{m} \frac{g_{m}^\omega}{\lb m \rb^{1+\alpha}} e^{ i m \cdot x - i |m|^2 t} \, .
\end{equation}
We can obtain a similar expression for $w(II)$,
namely
\begin{equation}\label{BouLinRepr2}
 w (II)(x, t) = \int \phi(\lambda) \sum_{m} b_{\lambda}(m) e^{ i m \cdot x - i |m|^2 t} \, d \lambda \,,
\end{equation}
where $\phi$ satisfies 
\begin{equation}\label{FinallyPhi}
\int |\phi (\lambda)| \, d \lambda \lesssim \|  w(II) \|_{X^{\alpha + \sigma, \frac12+}} \,,
\end{equation}
and the coefficients $b_{\lambda}(m)$ satisfy 
\begin{equation}\label{RenormalizationB}
\sum_{m} \lb m \rb^{2 \alpha + 1 - } |b_{\lambda}(m)|^2 = 1 \, .
\end{equation}
To prove \eqref{BouLinRepr2}--\eqref{RenormalizationB} we change variables by setting $\tau' = \tau + |m|^2$:
\begin{align}\nonumber
&  w(II)(x,t) = \sum_{m} \int e^{i x \cdot m + i t \cdot \tau} \widehat{  w(II)}(m, \tau) \, d \tau  
\\ \nonumber
& 
= \sum_{m} \int e^{i t \cdot \tau'} e^{i m \cdot x - i |m|^2 t} \widehat{ w(II)}(m, \tau' - |m|^2) \, d \tau'
\\ \nonumber
& 
=  \int    \left( \sum_{\ell} \ell^{2\alpha + 2 \sigma} |\widehat{  w(II)}(\ell, \tau' - |\ell|^2)|^{2}  \right)^{\frac12}
 e^{i t \cdot \tau'}  \sum_{m}   e^{i m \cdot x - i |m|^2 t} b_{\tau'} (m)  \, d \tau',
\end{align}
where 
we have defined 
\begin{equation}\label{DefinitionB}
b_{\lambda} (m) 
:= \frac{ \widehat{  w(II)}(m, \lambda - |m|^2)} {\left( \sum_{\ell} \lb \ell \rb^{2\alpha + 2 \sigma} |\widehat{  w(II)}(\ell, \lambda - |\ell|^2)|^{2} \right)^{\frac12}} \, .
\end{equation}
Thus \eqref{BouLinRepr2} holds with 
$$
\phi (\lambda) := \left( \sum_{\ell} \lb \ell \rb^{2\alpha + 2 \sigma} |\widehat{  w(II)}(\ell, \lambda - |\ell|^2)|^{2}  \right)^{\frac12}
 e^{i t \cdot \lambda}
$$
Notice that \eqref{RenormalizationB} is immediate by the definition \eqref{DefinitionB}. The property \eqref{FinallyPhi}
follows by the Cauchy--Schwartz inequality and changing variables $\lambda' = \lambda - |\ell|^2$:
\begin{align}\nonumber
\int |\phi(\lambda)|  & \, d\lambda \leq \left( \int \frac{d\lambda}{ \langle \lambda \rangle^{1+} }\right)^{\frac12} 
\left( \langle \lambda \rangle^{1+} \lb \ell \rb^{2\alpha + 2 \sigma} |\widehat{  w(II)}(\ell, \lambda - |\ell|^2)|^{2} \, d \lambda \right)^{\frac12}
\\ \nonumber
&
\lesssim 
\left( \langle \lambda' + |\ell|^2 \rangle^{1+} \lb \ell \rb^{2\alpha + 2 \sigma} |\widehat{  w(II)}(\ell, \lambda'  ) \, d \lambda' \right)^{\frac12}
= \|  w(II) \|_{X^{\alpha + \sigma, \frac12+}} \, .
\end{align}
We now come back to the $u$ functions and introduce the notation
\begin{equation}
a_{u(J), \lambda}(m) := \begin{cases}
\frac{g_{m}^\omega}{\lb m \rb^{1+\alpha}} & \mbox{if }  J=I, \vspace{6pt}
\\
b_{\lambda} (m) & \mbox{if }  J=II \, .  
\end{cases}
\end{equation}
Recalling \eqref{BouLinRepr1} and \eqref{BouLinRepr2}, we have
\begin{align}\label{AfterAveraging}
 \operatorname{P}_N & \operatorname{P}_{\left\{ \lb \tau + |n|^2 \rb \leq N^{2s} \right\} }   
\mathcal N_1 (\operatorname{P}_{N_1} u_1 (I), \operatorname{P}_{N_2} u_2 (J_2), \operatorname{P}_{N_3} u_3 (J_3) )     
\\ \nonumber
&
=  \int    
 \operatorname{P}_N \operatorname{P}_{\left\{ \lb \tau + |n|^2 \rb \leq N^{2s} \right\} }  \left( \sum_{|n_j| \sim N_j}  e^{i x \cdot(n_1 -n_2 + n_3)}  e^{-it (|n_1|^2 - |n_2|^2 + |n_3|^2)} \right)
 \\ \nonumber
&      \qquad  \qquad  \qquad  
   \times \prod_{j=1, 2, 3}  a_{u_j(J_j), \lambda_j} (n_j) \delta_{J_j} \biggl( \phi(\lambda_j)   \, d \lambda_j \biggr)  \, ,
\end{align}
where 
$$
\delta_{J_j}\bigl( \phi(\lambda_j)   \, d \lambda_j \bigr)  = \begin{cases}
 1 & \mbox{if }    J_{j} = I,
\\
 \phi(\lambda_j)   \, d \lambda_j  & \mbox{if }    J_{j} = II.
\end{cases}
$$
Thus using Minkowski's inequality and recalling \eqref{FinallyPhi} 
we see that \eqref{FarFromParaboloid} satisfies the desired inequalities
\eqref{WWHTP1}, \eqref{WWHTP2}, \eqref{WWHTP3} as long as we can bound  
\begin{align}\label{BeforeRapDec}
& 
N^{2\alpha + 2\sigma} \left\| \sum_{N_1, N_2, N_3}   
     \operatorname{P}_N \operatorname{P}_{\left\{ \lb \tau + |n|^2 \rb \leq N^{\frac{11}{10}} \right\} } 
    \left( \sum_{|n_j| \sim N_j}  e^{i x \cdot(n_1 -n_2 + n_3)}  e^{-it (|n_1|^2 - |n_2|^2 + |n_3|^2)} \right)
     \right.
  \\ \nonumber
  &  \qquad \qquad \qquad  \qquad \qquad \qquad  \qquad \quad  \quad 
  \left.    \times
   \prod_{j=1, 2, 3}  a_{u_j(J_j), \lambda_j} \right\|^2_{X^{0, -\frac12++}}
  \lesssim  \left( - \ln \varepsilon \right)^3 N^{0-} \, ,
 \end{align}
uniformly in $\lambda_j$, for all $\varepsilon \in (0,1)$ sufficiently small, on a set of probability larger than $1 - \varepsilon$.
All the following estimates are indeed  uniform in $\lambda_j$ and the exceptional set on which 
\eqref{BeforeRapDec} could be not satisfied is independent of $\lambda_j$. We omit the subscript~$\lambda_j$ to simplify the notation.

Since 
\begin{align}
&  \mathcal{F}\Bigl(  e^{i x \cdot(n_1 -n_2 + n_3)} e^{-it (|n_1|^2 - |n_2|^2 + |n_3|^2)} \Bigr)(n, \tau)
\\ \nonumber
&
\qquad \qquad \qquad \qquad 
 = \sum_{n_1 - n_2 + n_3 = n} \delta (\tau + |n_1|^2 - |n_2|^2 + |n_3|^2)  \, ,
\end{align}
where $\mathcal{F}$ is the space-time Fourier transform
and $\delta$ is the delta function, we reduce~\eqref{BeforeRapDec} to showing that
\begin{multline}\label{IntegrationOverMu}
N^{2\alpha + 2\sigma} \sum_{N_1, N_2, N_3} 
\sum_{|n| \sim N }
\frac{ \chi_{ \{  \langle |n|^2 - |n_1|^2 + |n_2|^2 - |n_3|^2  \rangle \leq N^{\frac{11}{10}} \} } }{  \langle |n|^2 - |n_1|^2 + |n_2|^2 - |n_3|^2  \rangle^{1--}}
\\ 
\times
\left|
 \sum_{ \substack{|n_j| \sim N_j, \, n_2 \neq n_1, n_3 \\  n = n_1 -n_2 + n_3   }} 
\prod_{j=1, 2, 3} a_{u_j(J_j)}(n_j)  \right|^2
\lesssim  \left( - \ln \varepsilon \right)^3 N^{0-}. 
\end{multline}
Letting
$$
\mu =  |n|^2  - |n_1|^2 + |n_2|^2 - |n_3|^2 \, .
$$
we see that \eqref{IntegrationOverMu} follows by 
\begin{multline}\label{AfterIntegrationOverMuPreq}
N^{2\alpha + 2\sigma} \sum_{N_1, N_2, N_3} 
\sum_{\mu \in \Z, \langle \mu \rangle \leq N^{\frac{11}{10}} } \frac{1}{\langle \mu \rangle^{1--}}  
\sum_{|n| \sim N} \left|
 \sum_{ R_{n}(n_1, n_2, n_3) } 
\prod_{j=1, 2, 3} a_{u_j(J_j)}(n_j)  \right|^2
\\
\lesssim  \left( - \ln \varepsilon \right)^3   N^{0-}  \, ,
\end{multline}
where for fixed $n, \mu$ we have denoted 
\begin{align}\label{R_nDef}
 R_{n}(n_1, n_2, n_3) 
 : =  & \Big\{ (n_1, n_2, n_3) \in \mathbb{Z}^3 \, : \,
  |n_j| \sim N_j, j = 1, 2, 3,
   \\ \nonumber
 &  
 n_{2} \neq n_1, n_3, \, 
  n_1 -n_2 + n_3 = n, \, 
  \mu =  |n|^2  -  |n_1|^2 + |n_2|^2 - |n_3|^2 
  \Big\}  \, .
\end{align}
The set $R_{n}(\cdot)$ depends on $\mu$ also (like all the sets we define below). However we omit this dependence to simplify the notation.
Since summing $\langle \mu \rangle^{-1+}$ gives an $N^{0+}$ factor, we reduced to prove
\begin{multline}\label{AfterIntegrationOverMuPreq}
N^{2\alpha + 2\sigma} \sum_{N_1, N_2, N_3} 
\sup_{|\mu| \lesssim N^{\frac{11}{10}} }  
\sum_{|n| \sim N} \left|
 \sum_{ R_{n}(n_1, n_2, n_3) } 
\prod_{j=1, 2, 3} a_{u_j(J_j)}(n_j)  \right|^2
\\
\lesssim  \left( - \ln \varepsilon \right)^3   N^{0-}  \, ,
\end{multline}
Notice that in the definition of $R_{n}(\cdot)$ the condition 
$$
|n|^2  -  |n_1|^2 + |n_2|^2 - |n_3|^2  = \mu 
$$
can be equivalently replaced by
$$
 2 (n_1 - n_2 ) \cdot (n_3 - n_2) = \mu \, .
 $$
Recalling that $M_1, M_2, M_3$ is the decreasing order of $N_1, N_2, N_3$, we now notice that we must 
have $N_1 \sim M_1$ or $N_3 \sim M_1$. Indeed, if we assume that both $N_1 \ll M_1$ and $N_3 \ll  M_1$ we must 
have $N_2 \sim M_1 \sim N$ and $\mu \sim N^2$, which contradicts the fact that $ \mu \sim N^{\frac{11}{10}}$. 
Since the roles of $N_1$ and  $N_3$ are symmetric (they are always the size of the indices of the Fourier coefficents of $u_1, u_3$), hereafter we assume that
$$N_1 = M_1 \sim N
\qquad  \mbox{and so $u_1 = w_1$} \, ;
$$ 
recall that $w_1$ is the $u_j$ supported on the largest frequency, and we have previously 
reduced to considering $w_{1}(J_1) = w_1(I)$; see \eqref{TheLargeIsRandom}. Thus, the argument above allows us to further reduce  
\eqref{AfterIntegrationOverMuPreq} to showing that
\begin{equation}\label{AfterIntegrationOverMuFinal}
N^{2\alpha + 2\sigma} 
\sum_{N_1, N_2, N_3}
\sup_{|\mu| \lesssim N_1^{\frac{11}{10}}}  
\sum_{|n| \sim N_1} \left|
 \sum_{R_{n}(n_1, n_2, n_3)} 
\frac{g_{n_1}^\omega}{\lb n_1 \rb^{1+ \alpha}} a_{u_2(J_2)}(n_2) a_{u_3(J_3)}(n_3)   \right|^2
\lesssim \left( - \ln \varepsilon \right)^3    N^{0-}  \, .
\end{equation}

To estimate \eqref{AfterIntegrationOverMuFinal} we can now distinguish few last possibilities.
It is useful to denote 
\begin{align}\label{SDef}
 S(n_1, n_2, n_3) 
 : =  & \Big\{ (n_1, n_2, n_3) \in \mathbb{Z}^3 \, : \,
  |n_j| \sim N_j, j = 1, 2, 3,
   \\ \nonumber
 &  
 n_{2} \neq n_1, n_3, \, 
   \mu =  2 (n_1 - n_2 ) \cdot (n_3 - n_2) 
  \Big\}  \, .
\end{align}

\subsubsection*{Case $J_2 =  J_3 = I$}
We must show  that
\begin{equation}\label{AllRandom}
N^{2\alpha + 2\sigma}  
\sum_{N_1, N_2, N_3}
\sup_{|\mu| \lesssim N_1^{\frac{11}{10}}}  
\sum_{|n| \sim N_1}
\left|
 \sum_{R_{n}(n_1, n_2, n_3)} 
\frac{g_{n_1}^\omega}{\lb n_1 \rb^{1+\alpha}} 
\frac{ \overline{ g_{n_2}^\omega}}{\lb n_2 \rb^{1+\alpha}} 
\frac{g_{n_3}^\omega}{\lb n_3 \rb^{1+\alpha}}   \right|^2
\lesssim  \left( - \ln \varepsilon \right)^3  N^{0-}  \, .
\end{equation}
Recalling that 
\eqref{HypercontractivityFinal}
along with the fact that the sum is restricted over $n_1, n_3 \neq n_2$ and symmetric under $n_1 \leftrightarrow n_3$, we get
\begin{equation}\nonumber 
 \int
\left|
 \sum_{ R_{n}(n_1, n_2, n_3) }
 \frac{g_{n_1}^\omega}{\lb n_1 \rb^{1+\alpha}} 
\frac{ \overline{ g_{n_2}^\omega}}{\lb n_2 \rb^{1+\alpha}} 
\frac{g_{n_3}^\omega}{\lb n_3 \rb^{1+\alpha}} 
 \right|^2 d \mathbb{P}(\omega)
 = 2
 \sum_{ R_{n}(n_1, n_2, n_3)}
 \frac{1}{\lb n_1 \rb^ {2 \alpha +2}} 
\frac{ 1 }{\lb n_2 \rb^{2 \alpha +2}} 
\frac{1}{\lb n_3 \rb^{2 \alpha +2}} \, .
\end{equation}
In the following bound we first restrict the summation over $(n_1, n_2,n_3) \in R_n(n_1, n_2,n_3)$ such 
that $n_1 \neq n_3$ (with a small abuse of notation we do not introduce additional notation
for this restriction). In this 
case
\begin{align}\label{AfterHyperc}
& \int  \sum_{|n| \sim N_1}
  \left| \sum_{R_{n}(n_1, n_2, n_3)} 
\frac{g_{n_1}^\omega}{\lb n_1 \rb^{1+\alpha}} 
\frac{ \overline{ g_{n_2}^\omega}}{\lb n_2 \rb^{1+\alpha}} 
\frac{g_{n_3}^\omega}{\lb n_3 \rb^{1+ \alpha}} \right|^2  d \mathbb{P}(\omega)
\\ \nonumber
&
\lesssim 
\sum_{|n| \sim N_1} 
 \sum_{ R_{n}(n_1, n_2, n_3)}
 \frac{1}{\lb n_1 \rb^ {2 \alpha + 2}} 
\frac{ 1 }{\lb n_2 \rb^{2 \alpha +2}} 
\frac{1}{\lb n_3 \rb^{2 \alpha + 2}}  
\\ \nonumber
&
\lesssim 
 \sum_{ S(n_1, n_2, n_3)}
 \frac{1}{\lb n_1 \rb^ {2 \alpha +2 }} 
\frac{ 1 }{\lb n_2 \rb^{2 \alpha +2}} 
\frac{1}{\lb n_3 \rb^{2 \alpha +2}}
\\ \nonumber
&
\sim
 \sum_{S(n_1, n_2, n_3)}
 N_1^{-2 \alpha -2 } N_2^{-2\alpha -2} N_3^{-2\alpha-2 }
 \\ \nonumber
&
\lesssim 
 N_1^{-2 \alpha -2  } N_2^{-2\alpha -2} N_3^{-2\alpha -2}
 \# S(n_1, n_2, n_3) 
 \\ \nonumber
&
\lesssim 
 N_1^{-2 \alpha -1  } N_2^{-2\alpha} N_3^{-2\alpha}
    \, ,
 \end{align}
where we used that if $n_1 \neq n_3$, then
$$
 \# S(n_1, n_2, n_3) \lesssim N_1 N_2^2 N_3^2 \, ;
$$
this is because once we have fixed $n_2, n_3$ in $N_2^2 N_3^2$ possible ways, we remain with at most $N_1$ choices for $n_1$
by the relation $\mu = 2 (n_1 - n_2 ) \cdot (n_3 - n_2)$. 
If we sum over $(n_1, n_2,n_3) \in R_n(n_1, n_2,n_3)$ such 
that $n_1 = n_3$, the restriction $\mu = 2 |n_1 - n_2|^2$ implies that once we have chosen $n_2$ in $N_{2}^2$
possible ways, we remain with $\lesssim \mu^{0+} \lesssim N_1^{0++}$ choices for $n_1 = n_3$ (since a circle of radius $\mu$ contains $\lesssim \mu^{0+}$ integer points).
This gives an even better bound than the one above. 
Summing the \eqref{AfterHyperc} over $N_2, N_3$ and recalling that $N_1 \sim N$, we have 
bounded the $L^2_{\omega}$ norm of the left hand side of \eqref{AllRandom}
by
$$
N^{2\alpha + 2\sigma} \sum_{N_1}  N_1^{-2 \alpha -1 } 
\lesssim
N^{2 \sigma - 1} \lesssim N^{0-}  \, ,
$$
where we used $\sigma < \frac12$.
Using the hypercontractivity of the Gaussian variables, we can upgrade this to an $L^{p}_{\omega}$ bound, for any $p < \infty$, 
with a constant $Cp^{3/2}$ (see \cite[Proposition 4.5]{MR3518561} for details). Proceeding as in the proof of 
Lemma \ref{LDBMink}, this 
implies \eqref{AllRandom} 
for all $\omega$ outside a set of probability smaller than $\varepsilon$, as required.

\subsubsection*{Case $J_2 =  J_3 = II$}

We show that
\begin{equation}\label{AllDeterministic}
N^{2\alpha + 2\sigma} 
\sum_{N_1, N_2, N_3}\sup_{|\mu| \lesssim N_1^{\frac{11}{10}}}  
\sum_{|n| \sim N_1} \left|
 \sum_{R_{n}(n_1, n_2, n_3)} 
\frac{g_{n_1}^\omega}{\lb n_1 \rb^{1+\alpha}} 
b_2(n_2) b_3(n_3)  \right|^2
\lesssim  \left( - \ln \varepsilon \right) N^{0-}  \, ,
\end{equation}
which clearly implies \eqref{AfterIntegrationOverMuFinal}.
We denote
$$
 R_{n, n_2}(n_1,n_3) := \{ (n_1, n_3) \in \mathbb{Z}^2 : (n_1, n_2, n_3) \in R_{n}(n_1, n_2, n_3)   \} \, ,
$$
and for $j=2,3$
\begin{equation}\nonumber 
   \| b_j \|_{\ell^{2}_{N_j}}^2 := \sum_{|n_j| \sim N_j}  |b_{j} (n_j)|^2 \, .
\end{equation}
Notice that by \eqref{RenormalizationB} (and recalling the change in notations) we have for $\sigma < 1/2$
\begin{equation}\label{GRenormalization}
\sum_{N_j}   N_j^{2\alpha + 2\sigma}   \| b_j \|_{\ell^{2}_{N_j}}^2 \lesssim 1,
\qquad
\sum_{N_j \lesssim N}  N_j^{2\alpha +1}   \| b_j \|_{\ell^{2}_{N_j}}^2 \lesssim N^{0+}  \, .
\end{equation}
Hereafter all the sums over indexes $n_j$ are restricted to $n_j \sim N_j$. We omit this fact in the subscripts to simplify the notation. 
We estimate
\begin{align}\label{SumoverN+b}
   \left| \sum_{R_{n}(n_1, n_2, n_3)} 
\frac{g_{n_1}^\omega}{\lb n_1 \rb^{1+\alpha}} 
b_{2}(n_2) b_{3}(n_3) \right|^2
 & 
 \leq 
\sum_{n_2}  |b_{2} (n_2)|^2 
\sum_{n_2} \left| \sum_{R_{n, n_2}(n_1,n_3)}  \frac{g_{n_1}^\omega}{\lb n_1 \rb^{1+\alpha}} 
b_{3}(n_3) 
\right|^2  
\\ \nonumber
&   \lesssim \| b_2 \|_{\ell^{2}_{N_2}}^2  \sum_{n_2}
\left| 
\sum_{R_{n, n_2}(n_1,n_3)}  \frac{g_{n_1}^\omega}{\lb n_1 \rb^{1+\alpha}} 
b_{3}(n_3) \right|^2 \, ,
\end{align}
where 
we have used the Cauchy--Schwartz inequality with respect to $n_2$ and \eqref{RenormalizationB}.
We further denote 
$$
 S_{n_2}(n_1,n_3) := \{ (n_1, n_3) \in \mathbb{Z}^2 : (n_1, n_2, n_3) \in S(n_1, n_2, n_3)   \} \, .
$$
We recall the estimate
\begin{equation}\label{GeomBound1}
\#  S_{n_2}(n_1,n_3) \lesssim N_{1}^{0+}
\qquad
\mbox{(Lemma 1 part (i) in \cite{bourgain1996invariant})} \,. 
\end{equation}
Thus, summing the \eqref{SumoverN+b} over $|n| \sim N_1$ yields
\begin{align}\label{Plugging1GeomBound1}
& \sum_n   \left| \sum_{R_{n}(n_1, n_2, n_3)} 
\frac{g_{n_1}^\omega}{\lb n_1 \rb^{1+\alpha}} 
b_{2}(n_2) b_{3}(n_3) \right|^2
\\ \nonumber
&
\lesssim  \| b_2 \|_{\ell^{2}_{N_2}}^2  \sum_{n, n_2}
\left| 
\sum_{R_{n, n_2}(n_1,n_3)}  \frac{g_{n_1}^\omega}{\lb n_1 \rb^{1+\alpha}} 
b_{3}(n_3) \right|^2
\\ \nonumber
&
\lesssim  \left( - \ln \varepsilon \right) \| b_2 \|_{\ell^{2}_{N_2}}^2  \sum_{n_2}
\left( 
\sum_{S_{n_2}(n_1,n_3)}   \frac{|b_{3}(n_3)|}{\lb n_1 \rb^{1+\alpha}} 
 \right)^2
\\ \nonumber
&  \lesssim  \left( - \ln \varepsilon \right)
\| b_2 \|_{\ell^{2}_{N_2}}^2  N_{1}^{0+} 
 \sum_{n_2}  \sum_{S_{n_2}(n_1,n_3)}   \frac{|b_{3}(n_3)|^2}{\lb n_1 \rb^{2\alpha +2}} 
 \\ \nonumber
& 
\lesssim  \left( - \ln \varepsilon \right)
 \| b_2 \|_{\ell^{2}_{N_2}}^2  N_{1}^{0+} N_{1}^{-2\alpha - 2} \sum_{S(n_1, n_2, n_3) }  
|b_{3}(n_3)|^2  \, ,
\end{align} 
where we used
\eqref{GeomBound1} 
 and the fact that 
$$
\sum_{S(n_1, n_2, n_3)} =
\sum_{n_2}  \sum_{S_{n_2}(n_1,n_3)}  \, .
$$
To justify the previous computation, in particular the factor of $(- \ln \varepsilon)$, we should first average over $d \mathbb{P}(\omega)$ and then use the 
hypercontractivity of the Gaussian variables. Since this works exactly as in the previous case ($J_2 =  J_3 = I$), we omit the details. We  do the same in \eqref{HypercontrAgain}. 
Denoting
$$
S_{n_3}(n_1, n_2) := \left\{ (n_1, n_2) \in \mathbb{Z}^3 \, : \, (n_1, n_2, n_3) \in S(n_1, n_2, n_3) \right\} \, ,
$$  
we recall   that
\begin{equation}\label{GeomBound2}
\# S_{n_3}(n_1, n_2) \lesssim N_1^{1+} N_2 
\qquad
\mbox{(Lemma 2 part (i) in \cite{bourgain1996invariant} switching $n_{1}$ and $n_3$)}
\, .
\end{equation}
Since
$$
 \sum_{S(n_1, n_2, n_3)}
= 
\sum_{n_3}
\sum_{S_{n_3}(n_1, n_2)}  \, ,
$$
we have by \eqref{GeomBound2} 
\begin{align}\label{Plugging2GeomBound1}
 \sum_{S(n_1, n_2, n_3) }  
|b_{3}(n_3)|^2 
&
\lesssim   
N_1^{1+} N_2 \sum_{n_3}    |b_3(n_3)|^2
\lesssim N_1^{1+} N_2  \| b_3 \|_{\ell^{2}_{N_3}}^2  \,. 
\end{align}
Plugging \eqref{Plugging2GeomBound1} into \eqref{Plugging1GeomBound1} we see that \eqref{AllDeterministic} is satisfied as long as
$$
N^{2 \alpha + 2\sigma} 
\sum_{N_1, N_2, N_3} N_1 ^{ - 2\alpha - 1 + 0+  }  N_2 \| b_2 \|_{\ell^{2}_{N_2}}^2   \| b_3 \|_{\ell^{2}_{N_3}}^2 
 \lesssim N^{0-} \, .
$$
Recalling \eqref{GRenormalization} and the fact that
$N \sim N_1 \gtrsim N_j$, $j=2, 3$, this is immediately verified for~$\sigma < \frac12$. 

\subsubsection*{Case $J_2 = I,  J_3 = II$}

We show that
\begin{equation}\label{RandomDeterministic}
N^{2\alpha + 2\sigma} 
\sum_{N_1, N_2, N_3}\sup_{|\mu| \lesssim N_1^{\frac{11}{10}}}  
\sum_{|n| \sim N_1} \left|
 \sum_{R_{n}(n_1, n_2, n_3)} 
\frac{g_{n_1}^\omega}{\lb n_1 \rb^{1+\alpha}} 
\frac{g_{n_2}^\omega}{\lb n_2 \rb^{1+\alpha}} 
 b_3(n_3)  \right|^2
\lesssim  \left( - \ln \varepsilon \right)^2 N^{0-}  \, ,
\end{equation}
which clearly implies \eqref{AfterIntegrationOverMuFinal}.
Since
$$
\#  R_{n}(n_1, n_2, n_3)
\lesssim N_2 N_3^{0+}
\qquad
\mbox{(Lemma 1 in \cite{bourgain1996invariant})}\, ,
$$
we can estimate using the Cauchy--Schwartz inequality: 
\begin{align}\label{HypercontrAgain}
 & 
 \left|
 \sum_{R_{n}(n_1, n_2, n_3)} 
\frac{g_{n_1}^\omega}{\lb n_1 \rb^{1+\alpha}} 
\frac{g_{n_2}^\omega}{\lb n_2 \rb^{1+\alpha}} 
 b_3(n_3)  \right|^2
\\ \nonumber
& 
\lesssim  \left( - \ln \varepsilon \right)^2
N_2 N_3^{0+}   
 \sum_{R_{n}(n_1, n_2, n_3)} 
\frac{|b_3(n_3)|^2}{ \lb n_1 \rb^{2+2\alpha} \lb n_2 \rb^{2+2\alpha}} 
   \\ \nonumber
 &
\lesssim  \left( - \ln \varepsilon \right)^2 N_2 N_3^{0+} N_{1}^{-2-2\alpha} N_{2}^{-2-2\alpha}
  \sum_{R_{n}(n_1, n_2, n_3)} 
 |b_3(n_3)|^2  
\\ \nonumber
&
\lesssim  \left( - \ln \varepsilon \right)^2
N_1^{-2-2\alpha} N_2^{-1-2\alpha} N_3^{0+}
  \sum_{R_{n}(n_1, n_2, n_3)} 
 |b_3(n_3)|^2 \, .
\end{align}
Summing this over $|n_1|\sim N_1$ yields
\begin{align}\label{Plugging1Geom2}
& \sum_{|n_1| \sim N_1} \left|
 \sum_{R_{n}(n_1, n_2, n_3)} 
\frac{g_{n_1}^\omega}{\lb n_1 \rb^{1+\alpha}} 
\frac{g_{n_2}^\omega}{\lb n_2 \rb^{1+\alpha}} 
 b_3(n_3)  \right|^2
 \\ \nonumber
 &
\lesssim   \left( - \ln \varepsilon \right)^2
N_1^{-2-2\alpha} N_2^{-1-2\alpha} N_3^{0+}
\sum_{|n_1| \sim N_1}  \sum_{R_{n}(n_1, n_2, n_3)} 
 |b_3(n_3)|^2 
 \\ \nonumber
 &
\lesssim   \left( - \ln \varepsilon \right)^2
N_1^{-2-2\alpha} N_2^{-1-2\alpha} N_3^{0+}
\sum_{S(n_1, n_2, n_3)} 
 |b_3(n_3)|^2  \, .
\end{align}
Then since
$$
 \sum_{S(n_1, n_2, n_3)}  
 = \sum_{n_3}
   \sum_{S_{n_3}(n_1, n_2)}    
$$
and
\begin{equation}\label{CardBefSymmetry}
\# S_{n_3}(n_1, n_2) \lesssim N_1^{1+} N_2,
\qquad 
\mbox{(Lemma 2 part (i) in \cite{bourgain1996invariant})}\, ,
\end{equation}
we have 
\begin{equation}\label{Plugging2Geom2}
 \sum_{S(n_1, n_2, n_3)} 
 |b_3(n_3)|^2 
 N_1^{1+} N_2 
 \sum_{n_3}   |b_3(n_3)|^2 
 \lesssim N_1^{1+} N_2    \| b_3 \|_{\ell^{3}_{N_3}}^2 \, .
\end{equation}
Plugging \eqref{Plugging2Geom2} into \eqref{Plugging1Geom2}
we see that the left hand side of \eqref{RandomDeterministic} is bounded by 
$$
N^{2\alpha + 2\sigma }
\sum_{N_1, N_2, N_3}   N_1^{- 2\alpha - 1 + 0+ } N_2^{-2\alpha} N_{3}^{0+}  \| b_3 \|_{\ell^{3}_{N_3}}^2 \lesssim 
N^{0-} 
$$
as required, where we used $\sigma <\frac12$, \eqref{GRenormalization}, and the fact that
$N \sim N_1 \gtrsim N_3$.

\subsubsection*{Case $J_2 = II,  J_3 = I$}
We proceed exactly as in the case 
$J_2 = I,  J_3 = II$, but we exchange the roles of $n_2$ and $n_3$. Notice that everything works symmetrically
under $n_2 \leftrightarrow n_3$ except the fact that the sets $S_{n_3}(n_1, n_2)$ and $S_{n_2}(n_1, n_3)$ do not coincide. However, in the previous argument, 
we only needed the
estimate \eqref{CardBefSymmetry}. Here we instead use    
$$
\# S_{n_2}(n_1, n_3) \lesssim N_1^{1+} N_3,
\qquad 
\mbox{(Lemma 2 part (ii) in \cite{bourgain1996invariant}) } \, ,
$$
whose right hand side is indeed the same as that of \eqref{CardBefSymmetry} after interchanging $N_2 \leftrightarrow N_3$. This concludes the  
proof of \eqref{WWHTP1}, \eqref{WWHTP2}, \eqref{WWHTP3} in dimension $d=2$.

The case $d=1$ is much easier. One 
can easily check that the previous argument indeed adapts and simplifies. 

\end{proof}

\subsubsection*{The quintic NLS on $\T$}

Here we explain how one can prove an analogous $\frac{1}{2}-$ smoothing result for the quintic NLS ($p=5$) on $\T$, 
after removing certain bad resonances from the nonlinearity, as we have done using the Wick order in the cubic case. We plan to 
study this problem in detail
in a
future work.
We consider\footnote{For more information about why this is the relevant nonlinear term to consider in the quintic case, consult~\cite{NS15}.} 
\begin{equation}\label{WOQuinticNonlinearity}
\mathcal{N}(u) := \pm u \left( |u|^5 - 3 \mu \right), 
\qquad 
\mu = \fint_{\T} |u(x, t)|^4 dx \, ,
\end{equation}
and
\begin{equation}\label{eq:wickNLSQuintic}
\begin{cases}
i \partial_t u + \Delta u = \mathcal{N}(u), \quad x\in\T,\\
u(x,0) = f^{\omega}(x),
\end{cases}
\end{equation}
with randomized initial data 
\begin{equation}\label{QuinticRandInitData}
f^\omega(x) = \sum_{n \in \mathbb{Z}} \frac{g_n^\omega}{\lb n \rb^{\frac{1}{2}+ \alpha}} e^{i n \cdot x} \, .
\end{equation} 
Recall that such data is $\mathbb{P}$-almost surely 
in $H^s$ for all $s < \alpha$ (namely we work at $H^{0+}$ level)  
and satisfies a uniform bound for these $H^{s}$ norms with arbitrarily high probability; see \eqref{HyperInitialdata}.
Proceeding as for the cubic equation before, and focusing only on the fully random evolution, 
namely the case $J_j = I$ for $j=1, \ldots, 5$, we reduce to proving the following fact. We fix 
$$
0<\sigma <  \frac12 .
$$ 
Then 
with probability at least $1 - \varepsilon$ , we have
\begin{multline}\label{AllRandomQuintic}
N^{2\alpha + 2 \sigma}
\sum_{N_1, \ldots, N_5}
\sup_{\mu}  
\sum_{|n| \sim N_1}
\Big|
 \sum_{R_{n}(n_1, \ldots, n_5)} 
\frac{g_{n_1}^\omega}{\lb n_1 \rb^{\frac12+\alpha}} 
\frac{ \overline{ g_{n_2}^\omega}}{\lb n_2 \rb^{\frac12+\alpha}} 
\frac{g_{n_3}^\omega}{\lb n_3 \rb^{\frac12+\alpha}}
\\
\times
 \frac{ \overline{ g_{n_4}^\omega}}{\lb n_4 \rb^{\frac12+\alpha}} 
\frac{g_{n_5}^\omega}{\lb n_5 \rb^{\frac12+\alpha}} 
   \Big|^2
\lesssim  \left( - \ln \varepsilon \right)^5  N^{0-}  \, ,
\end{multline}
where 
\begin{align}
 R_{n}(n_1, \ldots, n_5) 
 : =  & \Big\{ (n_1, \ldots, n_5) \in \mathbb{Z}^3 \, : \,
  |n_j| \sim N_j, j = 1, \ldots, 5,
   \\ \nonumber
 &  
 n_{2}, n_4 \neq n_1, n_3, n_5 \, ,
  n_1 -n_2 + n_3 - n_4 + n_5 = n, \, 
   \\ \nonumber
 &  
  \mu =  |n|^2  - |n_1|^2 + |n_2|^2 - |n_3|^2 + |n_4|^2 - |n_5|^2
  \Big\}  \, 
\end{align}
and we have assumed $N_1 = \max\{N_1, \ldots, N_5\}$, so that we can restrict to the case $N \sim N_1$. 
However, the argument below adapts immediately to the case in which the largest frequency 
is $N_2$ (all the other cases are clearly symmetric). Again, averaging in $d \mathbb{P}(\omega)$ and upgrading the corresponding estimate to any $L^{p}_{\omega}$, 
$p \in [2,\infty)$ by hypercontractivity, we reduce to proving, uniformly over $\mu$, the following:
\begin{align}\label{FINE}
& 
N^{2\alpha + 2 \sigma}
\sum_{N_1, \ldots, N_5}
\sum_{|n| \sim N_1}
 \sum_{R_{n}(n_1, \ldots, n_5)} 
\prod_{j=1, \ldots, 5} \frac{1}{\lb n_j \rb^{1+ 2\alpha}} 
\lesssim    N^{0-}  \, .
\end{align}

In fact, we have
\begin{align}\nonumber
& 
N^{2\alpha + 2 \sigma}
\sum_{N_1, \ldots, N_5}
\sum_{|n| \sim N_1}
 \sum_{R_{n}(n_1, \ldots, n_5)} 
\prod_{j=1, \ldots, 5} \frac{1}{\lb n_j \rb^{1+ 2\alpha}} 
\\ \nonumber
&\lesssim   N^{2\alpha + 2 \sigma}
\sum_{N_1, \ldots, N_5} 
\#R(n_1, \ldots, n_5)N_1^{-(1+ 2\alpha)}N_2^{-(1+ 2\alpha)}N_3^{-(1+ 2\alpha)}N_4^{-(1+ 2\alpha)}N_5^{-(1+ 2\alpha)}
\end{align}
and since  $\#R(n_1, \ldots, n_5)\lesssim N_5N_4N_3N_2$ and $\sigma < \frac12$, the estimate \eqref{FINE} is proved.

\subsection{The Cubic NLS Equation on $\R^d$ ($d=1,2$) with Random Data (Theorem \ref{MainTHM3})}\label{NLSon R^2}

\

We prove Theorem \ref{MainTHM3}. Given $f \in H^{s}(\R^d)$ with $s >0$, we are considering the randomized initial 
data $f^{\omega}$ defined in 
\eqref{eq:RdRand}. Remember that these functions are typically more integrable than $f$. They are $\mathbb{P}$-almost surely in $L^{p}$ for any $p \in [2, \infty)$.
On the other hand, they are not more regular than $f$, but they rather have comparable $H^s$ norms; see \eqref{HsRandomizNormOnR}.
We approximate the equation (here $\mathcal N(z) = \pm |z|^2z$) as in \eqref{TruncatedNLS}, for all $N \in 2^{\mathbb{N}} \cup \{ \infty \}$, and 
$\Phi^{N}_t f^{\omega}$ denotes the associated flow, with initial datum $f^{\omega}$. 

\subsubsection*{Proof of Theorem \ref{MainTHM3}}
Notice that \eqref{r2prob-lin2} is the content of Proposition \ref{RandomLinCOnv2}. To prove \eqref{r2prob-nonlin}, proceeding exactly as in the proof 
of 
Theorem \ref{MainTHM2},
it suffices that given any $\delta >0$ sufficiently small we prove
\begin{equation}\label{NonlinearMaxEstOnR}
\lim_{N \to \infty} \left\| \sup_{0 \leq t \leq \delta} | \Phi_{t} f^{\omega} (x) - \Phi^{N}_{t} f^{\omega} (x) | \right\|_{L^{2}_x(\T^2)} = 0
\end{equation}  
for all $f^{\omega}$ outside a $\delta$--exceptional set $A_\delta$. This can be done using Corollary \ref{uuu} exactly as in the proof of \eqref{NonlinearProbMaxEst}
using Proposition \ref{Bourgain1/2}. In fact, since Corollary \ref{uuu} is a deterministic statement, we can actually prove \eqref{NonlinearMaxEstOnR} 
for all $\omega$.  
To do so we need to require (compare with \eqref{CompareWithThisS})
\begin{equation}\label{CompareWithThisSBis} 
s_{\R^d} = \frac{d}{2(d+1)} < s + \sigma  ,
\qquad 
d=1, 2 \, .
\end{equation}
Since we used Corollary \ref{uuu}, we are allowed to take $\sigma < \min (2s, 1)$ and we see that \eqref{CompareWithThisSBis} 
is satisfied for all $s > \frac{d}{6(d+1)}$. This completes the proof.

\hfill $\Box$


\end{document}